\documentclass[10.5pt,reqno]{amsart}
\usepackage{longtable} 
\usepackage{hyperref}
\usepackage[T1]{fontenc}
\usepackage[utf8]{inputenc}
\usepackage[english]{babel} 
\usepackage{textcomp}
\usepackage{dsfont} 
\usepackage{latexsym}
\usepackage{amssymb}
\usepackage{amsthm}
\usepackage{amsmath}
\DeclareMathAlphabet{\mathpzc}{OT1}{pzc}{m}{en}
\usepackage{yfonts}
\usepackage{xfrac}
\usepackage{newlfont}
\usepackage{graphicx}
\usepackage{mathtools}
\usepackage{comment}
\usepackage{indentfirst}
\usepackage{braket}
\usepackage{mathrsfs}
\usepackage{xcolor}

\usepackage{etoolbox}

\usepackage{scalerel}[2014/03/10]
\usepackage[usestackEOL]{stackengine}
\newcommand{\dashint}{\,\ThisStyle{\ensurestackMath{%
			\stackinset{c}{.2\LMpt}{c}{.5\LMpt}{\SavedStyle-}{\SavedStyle\phantom{\int}}}%
		\setbox0=\hbox{$\SavedStyle\int\,$}\kern-\wd0}\int}

\newcommand{\Car}{\mathrm{C}}
\newcommand{\RC}{\mathrm{RC}}
\newcommand{\Samp}{\mathrm{S}}

\DeclareMathOperator{\card}{Card}

\DeclareMathOperator{\supp}{Supp}

\DeclareMathOperator{\ad}{ad}

\DeclareMathOperator{\essmin}{ess\,min}

\newcommand{\ee}{\mathrm{e}}

\newcommand{\loc}{\mathrm{loc}}
\newcommand{\vect}[1]{\mathbf{{#1}}}
\newcommand{\dd}{\mathrm{d}}

\DeclarePairedDelimiter{\abs}{\lvert}{\rvert}

\DeclarePairedDelimiter{\norm}{\lVert}{\rVert}

\let\originalleft\left
\let\originalright\right
\renewcommand{\left}{\mathopen{}\mathclose\bgroup\originalleft}
\renewcommand{\right}{\aftergroup\egroup\originalright}

\newcommand{\grado}{\Df}
\newcommand{\N}{\mathds{N}}
\newcommand{\Z}{\mathds{Z}}
\newcommand{\Q}{\mathds{Q}}

\newcommand{\C}{\mathds{C}}

\newcommand{\R}{\mathds{R}}

\newcommand{\gf}{\mathfrak{g}}

\newcommand{\Df}{\mathfrak{D}}

\newcommand{\Bs}{\mathscr{B}}
\newcommand{\Cc}{\mathcal{C}}
\newcommand{\Dc}{\mathcal{D}}

\newcommand{\Fc}{\mathcal{F}}
\newcommand{\Gc}{\mathcal{G}}

\newcommand{\Ic}{\mathcal{I}}

\newcommand{\Lc}{\mathcal{L}}
\renewcommand{\Mc}{\mathcal{M}}

\newcommand{\Pc}{\mathcal{P}}

\newcommand{\Rc}{\mathcal{R}}
\newcommand{\Sc}{\mathcal{S}}

\newcommand{\Fs}{\mathscr{F}}

\newcommand{\meg}{\leqslant}
\newcommand{\Meg}{\geqslant}
\newcommand{\eps}{\varepsilon}
\renewcommand{\phi}{\varphi}
\newcommand{\mi}{\mu}

\thanks{{\em Math Subject Classification 2020}: 46E36, 22E30.}
\thanks{The author is a member of the 	Gruppo Nazionale per l'Analisi
	Matematica, la Probabilit\`a e le	loro Applicazioni (GNAMPA) of
	the Istituto Nazionale di Alta Matematica (INdAM). The author was partially funded by the INdAM-GNAMPA Project CUP\_E5324001950001.
}

\begin{document}
	\title[Besov and Triebel--Lizorkin Spaces]{Besov and Triebel--Lizorkin Spaces on Filtered Lie Groups, III: the Spaces $F^{\infty,q}_\alpha$}
	
	\author[M.\ Calzi]{Mattia Calzi} 
	\address{Dipartimento di Matematica, Universit\`a degli Studi di
		Milano, Via C. Saldini 50, 20133 Milano, Italy}
	\email{{\tt mattia.calzi@unimi.it}}
	
	\theoremstyle{definition}
	\newtheorem{deff}{Definition}[section]

	\newtheorem{oss}[deff]{Remark}
	
	\newtheorem{ass}[deff]{Assumptions}
	
	\newtheorem{nott}[deff]{Notation}

	\theoremstyle{plain}
	\newtheorem{teo}[deff]{Theorem}
	
	\newtheorem{lem}[deff]{Lemma}
	
	\newtheorem{prop}[deff]{Proposition}
	
	\newtheorem{cor}[deff]{Corollary}
	
	\begin{abstract}
		We continue to develop a theory of  (Besov and) Triebel--Lizorkin spaces associated with weighted subcoercive operators on a real connected Lie group, focusing on the Triebel--Lizorkin space $F^{\infty,q}_\alpha(\beta)$, which requires a more technical treatment.
	\end{abstract}
 
 \maketitle

\section{Introduction}

Besov and Triebel--Lizorkin spaces form a large class of function spaces on the Euclidean spaces which provides a uniform, albeit somewhat technical, way to study simultaneously several classical function spaces, such as Sobolev spaces with integer regularity, fractional Sobolev spaces, both in the version of Sovolev--Slobodeckij--Gagliardo spaces and in the version of Bessel potential spaces, Lipschitz spaces, Hardy and BMO spaces, etc. These spaces have been extensively studied in the classical Euclidean setting (cf., e.g.,~\cite{TriebelFS,TriebelFS2,TriebelFS3}), but have also been extended to more general contexts, such as: open subsets of $\R^n$ (cf., e.g.,~\cite[Chapter 5]{TriebelFS2}); Riemannian manifolds with bounded geometry (cf., e.g.,~\cite[Chapter 7]{TriebelFS2}); Lie groups endowed with a left-invariant Riemannian (cf., e.g.,~\cite[Chapter 7]{TriebelFS2}) or sub-Riemannian metric (cf., e.g.,~\cite{BPV,BPV2,BPV3}); metric spaces endowed with suitable operators resembling a (sub-)Laplacian (cf., e.g.,~\cite{TriebelFS3,Hu}).\footnote{The literature on the subject is quite extensive and the above mentioned reference should only be intended as a very short list of examples, which is by no means complete.} 
There are nonetheless some contexts where `measuring regularity' in a `Riemannian way,' that is, grouping together all differential operators of the same order, or, more generally, in a `sub-Riemannian way,' appears to be inconvenient since it either clashes with the underlying geometry of the space or with the structure of the differential operator at hand. For instance, let $G$ be a homogeneous group, that is, a simply connected nilpotent Lie group whose Lie algebra $\gf$ has a graduation $(\gf_\lambda)_{\lambda>0}$; in other words, $\gf=\bigoplus_{\lambda>0} \gf_\lambda$ and $[\gf_\lambda,\gf_\nu]\subseteq \gf_{\lambda+\mi}$ for every $\lambda,\mi>0$. Then, $G$ may be identified with $\gf$ by means of the exponential map and $\gf$ may be endowed with a family of automorphic dilations $(\delta_r)_{r>0}$ defined so that $\delta_r(X)=r^\lambda X$ for every $X\in \gf_\lambda$. Operators which are compatible with these dilations (that is, homogeneous operators) are consequently quite natural in this context and one is therefore led to consider (`Goodman type') Sobolev spaces of the form $\Set{f\in L^p(G)\colon\forall \alpha\; (d_\alpha\meg k \implies\vect X^\alpha f\in L^p(G))}$, where $\vect X^\alpha=X_1^{\alpha_1}\cdots X_n^{\alpha_n}$ for some homogeneous basis $(X_1,\dots, X_n)$ of $\gf$, and where $d_\alpha=\sum_j \alpha_j \deg(X_j)$. It turns out, however, that these spaces behave quite weirdly for general $k$ -- for instance, they do not interpolate as one may expect. In fact, a different class of `Bessel potential' Sobolev spaces exhibiting a more natural behviour was introduced in~\cite{FischerRuzhansky}, and was shown to coincide with the previous `Goodman type' Sobolev spaces only for specific values of $k$. It is now worthwhile remarking that, whereas the usual `Bessel potential' Sobolev spaces (and, more generally, several of the Besov and Triebel--Lizorkin spaces briefly mentioned above) are essentially constructed using a (sub-)Laplacian, these `homogeneous' Sobolev spaces were constructed using positive Rockland operators instead, that is, homogeneous and hypoelliptic left-invariant differential operators. A definition using second order differential operator would simply not be possible. 
Notice, by the way, that replacing a second-order subelliptic differential operator with a higher-order one provides several technical difficulties, since the associated heat kernel cannot be positive, and there is no longer any finite speed property for the corresponding wave propagator -- tools which often lie at the core of several proofs in the literature.

It is therefore natural to wonder whether there is a more general framework which allows to deal at the same time with sub-Laplacians and positive Rockland operators. As a matter of fact, ter Elst and Robinson showed that weighted subcoercive operators provide a quite reasonable and natural choice (cf.~\cite{ElstRobinson}). Indeed, given a group $G$ whose Lie algebra is endowed with a suitable increasing filtration $(\gf_\lambda)_{\lambda>0}$, it is possible to associated a homogeneous group $G_*$ (its `contraction') to $G$ in a natural way, and to associate to every left-invariant differential operator $\Lc$ of degree $d$ some homogeneous left-invariant differential operator $P$ of degree $d$ on $G_*$ (which plays the r\^ole of the `principal part' of $\Lc$). The operator $\Lc$ is then said to be weighted subcoercive if $P+P^*$ is a positive Rockland operator. Notice that this definition mimics closely that of elliptic operators, and that Rockland operators play the r\^ole of homogeneous elliptic operators. Weighted subcoercive operators then enjoy several useful properties. For example, they generate a heat semigroup $(\ee^{-t\Lc})$ whose convolution kernel satisfies suitable Gaussian estimates (even though the exponential decay depends on the degree $d$ and is milder than the classical one). If, in addition, $\Lc$ is formally self-adjoint, then the closure of $\Lc$ on the space of test functions is self-adjoint on $L^2$, hence generates a functional calculus which enjoys particularly interesting properties when $G$ has polynomial growth.
As shown in~\cite{BCP}, using the heat semigroup associated with a weighted subcoercive operator allows one to define natural Besov and Triebel--Lizorkin  spaces  $B^{p,q}_\alpha$ and $F^{p,q}_\alpha$ on a general connected filtered Lie group. The resulting spaces then do not depend on the chosen operator. In~\cite{BCP,Calzi2} several properties of these spaces were proved, including:
\begin{itemize}
	\item $B^{p,q}_\alpha$ and $F^{p,q}_\alpha$ are Banach spaces;
	
	\item the space $C^\infty_c(G)$ of test functions is dense in $B^{p,q}_\alpha$ and $F^{p,q}_\alpha$ for $p,q<\infty$;
	
	\item $B^{p',q'}_{-\alpha}$ and $F^{p',q'}_{-\alpha}$ may be canonically identified with the duals of $B^{p,q}_\alpha$ and $F^{p,q}_\alpha$, respectively, when $p,q<\infty$;
	
	\item $B^{p_1,q_1}_{\alpha_1}\subseteq B^{p_2,q_2}_{\alpha_2}$ when $p_1\meg p_2$, $\alpha_2 -Q_*/p_2\meg \alpha_1-Q_*/p_1$, and either $q_1\meg q_2$ or $\alpha_2 -Q_*/p_2< \alpha_1-Q_*/p_1$, where $Q_*$ denotes the homogeneous dimension of $G_*$;\footnote{This and the following fact are actually true when $G$ is endowed with a left Haar measure, and should be slightly modified in the general case.}
	
	\item   $F^{p_1,q_1}_{\alpha_1}\subseteq F^{p_2,q_2}_{\alpha_2}$ when $p_1\meg p_2$, $\alpha_2 -Q_*/p_2\meg \alpha_1-Q_*/p_1$, and either $q_1\meg q_2$ or $p_1<p_2$;
	
	\item $B^{p,q}_\alpha,F^{p,q}_\alpha \subseteq L^p$ when $\alpha>0$, and $F^{p,2}_0=L^p$ for $p\in (1,\infty)$;
	
	\item if $\omega\in \R$ is sufficiently large, then $(\Lc+\omega I)^{-\alpha'}$ induces canonical isomorphisms of $B^{p,q}_\alpha$ and $F^{p,q}_\alpha$ onto $B^{p,q}_{\alpha+\alpha'}$ and $F^{p,q}_{\alpha+\alpha'}$, respectively;
	
	\item if $(X_j)$ is a family of elements of $\gf$ such that the corresponding elements $Y_j$ of $\gf_*$ induce a basis of $\gf_*/[\gf_*,\gf_*]$, and if $\dd$ is the least common multiple of the $d_j=\deg(X_j)$, then $f\in B^{p,q}_\alpha$ (resp.\ $f\in F^{p,q}_\alpha$) if and only if $\ee^{-\Lc}f\in L^p$ and $X_j^{\dd/d_j}\in B^{p,q}_{\alpha-\dd}$ (resp.\ $X_j^{\dd/d_j}\in F^{p,q}_{\alpha-\dd}$) for every $j$. In particular $F^{p,2}_{k\dd}$ coincides with the above-mentioned `Goodman type' Sobolev spaces when $k\in\N$ and $p\in (1,\infty)$;
	
	\item the spaces $B^{p,q}_\alpha$ and the spaces $F^{p,q}_\alpha$ interpolate as the classical ones;
	
	\item $B^{p,q}_\alpha\cap L^\infty$ and $F^{p,q}_\alpha\cap L^\infty$ are algebras under pointwise multiplication for every $\alpha>0$. In particular, $B^{p,q}_\alpha$ and $F^{p,q}_\alpha$ are algebras for every $\alpha>Q_*/p$;
	
	\item the spaces of pointwise multipliers of $B^{p,q}_\alpha$ and $F^{p,q}_\alpha$ may be characterized for $\alpha>Q_*/p$;
	
	\item  the spaces $F^{p,q}_\alpha$  enjoy a localization property as the classical ones;

	\item some Besov and Triebel--Lizorkin spaces may be described in terms of (finite) differences.
\end{itemize}

The purpose of this paper is to introduce a study the spaces $F^{\infty,q}_\alpha$, for $q\in [1,\infty]$ and $\alpha\in \R$, following the approach by Frazier and Jawerth~\cite{FrazierJawerth}. Since these spaces require a rather technical modification of the definition of the other Triebel--Lizorkin spaces, we preferred to avoid introducing them in~\cite{BCP,Calzi2} and deal with all the necessary technicalities here, proving the analogues of several of the above properties, or completing the picture in several of the above statements. It should be noticed, nonetheless, that we were not able to extend all the properties in this generality.

Let us also mention that several additional works (cf.~\cite{Calzi4,Calzi5,CalziRizzo}) on this subject are in preparation. In particular, we plan to   study the full scale of Besov and Triebel--Lizorkin spaces (that is, for $p,q\in (0,\infty]$) in the more particular case in which $G$ has polynomial volume growth. In this latter context, we plan to study also additional properties, such as discretization, and to study the relationship between Triebel--Lizorkin spaces and local Hardy and bmo spaces.

\smallskip

Here is a plan of the paper.  In Section~\ref{sec:2}, we shall collect several preliminary definitions and properties, mostly referring to~\cite{BCP} for proofs. In particular, we shall present a precise definition of Besov and Triebel--Lizorkin spaces in this context, except for the spaces $F^{\infty,q}_\alpha$. In Section~\ref{sec:3} we shall introduce the spaces $\Cc^q(\mi,a)$, which will play the r\^ole of the mixed-norm spaces $L^{q,p}(\mi,\beta)$ in the definition of $F^{p,q}_\alpha(\beta)$, and prove some its basic properties. In Section~\ref{sec:4} we prove some technical lemmas which eventually lead to  the fundamental Lemma~\ref{lem:8b}, which is essential to prove the equivalence of the various norms which will be used in Section~\ref{sec:5} to define the spaces $F^{\infty,q}_\alpha(\beta)$. In Section~\ref{sec:5} we shall finally consider the properties of Triebel--Lizorkin spaces studied in~\cite{BCP,Calzi2} and provide the analogues which include also the spaces $F^{\infty,q}_\alpha(\beta)$, when possible. We shall in fact only indicate the necessary modifications in the proofs of the corresponding results in~\cite{BCP,Calzi2}, except when a completely different proof is needed.

\section{Preliminaries}\label{sec:2}

\subsection{Relatively Invariant Measures and Convolution}

Throughout the paper, we shall denote with $G$ a connected (finite-dimensional, real) Lie group  with Lie algebra $\gf$. We shall endow $G$ with a non-trivial positive relatively invariant (Radon) measure $\beta$. Thus, there are two positive characters $\Delta_L$ and $\Delta_R$ of $G$ such that 
\[
\Delta_L(y)\int_G f(y x)\,\dd \beta(x) =  \int_G f(x)\,\dd \beta(x) = \Delta_R(y) \int_G f(x y )\,\dd \beta(x)
\]
for every $f\in L^1(\beta)$ and for every $y\in G$. In particular,
\[
\beta(x A)=\Delta_L(x) \beta(A) \qquad\text{and}\qquad \beta(A x)=\Delta_R(x) \beta(A)
\]
for every $\beta$-measurable subset $A$ of $G$ and for every $x\in G$. In addition,
\[
\int_G f(x^{-1}) \,\dd \beta(x)= \int_G f(x) \Delta_L(x^{-1})\Delta_R(x^{-1})\,\dd \beta(x)
\]
for every positive $\beta$-measurable function $f$ on $G$. We shall define $\beta_L\coloneqq \Delta_L^{-1}\cdot \beta$ and $\beta_R\coloneqq \Delta_R^{-1}\cdot \beta$, so that $\beta_L$ is a left Haar measure and $\beta_R$ is a right Haar measure.

We now recall some facts about convolution. Given two convolvable\footnote{We shall not provide a precise definition of this concept, since this would drag us too far away from the main topic. } distributions $f,g$ on $G$, one has
\[
\langle f*g, \phi\rangle= \langle f\otimes g, (x,y)\mapsto \phi(xy)\rangle
\]
for every $\phi\in C^\infty_c(G)$.
We shall identify each $f\in L^1_\loc(\beta)$ with $f\cdot \beta$, that is, the measure with density $f$ with respect to $\beta$. Given two convolvable functions $f,g$ such that $f*g$ is absolutely continuous with respect to $\beta$, we shall generally identify $f*g$ with its density with respect to $\beta$. Thus, under very mild conditions which will be always verified in the applications,
\begin{equation}\label{eq:1}
	(f* g)(x)=\int_G f(x y^{-1}) g(y) \Delta_R(y^{-1})\,\dd \beta(y)= \int_G f(y) g(y^{-1}x)\Delta_L(y^{-1})\,\dd \beta(y).
\end{equation} 

Similar formulae apply when either $f$ or $g$ is a distribution (and the convolution is still a function).
Observe that, if $f,g$ are convolvable distributions, $X$ is a left-invariant differential operator, and $Y$ is a right-invariant differential operator, then $Yf $ and $X g$ are convolvable and
\[
YX(f*g)=(Yf)*(Xg).
\]

If $f,g,h$ are distributions then, under some reasonable conditions that will always be satisfied in the applications, 
\[
\langle f * g, h \rangle =\langle f, h* (\Delta_R \Rc g)\rangle =\langle g, (\Delta_L \Rc f)* h\rangle,
\]
where $\langle \Rc f,\phi\rangle=\langle f,\check \phi \rangle$ for every $\phi\in C^\infty_c(G)$, and $\check \phi=\phi(\,\cdot\,^{-1})$. A word of caution here: $\Rc \beta= \Delta_L^{-1} \Delta_R^{-1} \beta$, so that $\Rc(f\cdot \beta) =(\Delta_L^{-1} \Delta_R^{-1}  \check f)\cdot \beta$. In other words, one must be careful not to confuse $\Rc f$ with $\check f$ when $f\in L^1_\loc(\beta)$.

We now recall Young's inequality in this context. Take $p_1,p_2,p_3\in [1,\infty]$ so that $\frac{1}{p_1'}+\frac{1}{p_2'}=\frac{1}{p_3'}$. Then,
\[
\norm{(\Delta_L^{1/p_2'} f)*(\Delta_R^{1/p_1'}g)}_{L^{p_3}(\beta)}\meg \norm{f}_{L^{p_1}(\beta)}\norm{g}_{L^{p_2}(\beta)}
\]
or, equivalently,
\[
\norm{f*g}_{L^{p_3}(\beta)}\meg \norm{\Delta_L^{-1/p_2'}f}_{L^{p_1}(\beta)}\norm{\Delta_R^{-1/p_1'}g}_{L^{p_2}(\beta)}
\]
for every two positive $\beta$-measurable functions $f,g$ (for positive measurable functions, convolution may be defined by means of~\eqref{eq:1}). See~\cite[Lemma 2.1]{KR78} when $\Delta_L=1$. The general case follows easily.

\subsection{Differential Operators}

\begin{deff}
	We shall generally identify the (complexification of the) enveloping algebra $U(G)$ of $\gf$ with the algebra of left-invariant differential operators. We shall fix a scalar product on $\gf$, and we shall endow $U(G)$ with the corresponding scalar product, namely the quotient of the natural scalar product on the (complexfication of the) tensor algebra over $\gf$.\footnote{The actual scalar product on $U(G)$ will essentially not matter in the sequel. This one is relatively convenient, since $\abs{X}=\abs{X^+}=\abs{\overline X}$ for every $X\in U(G)$.}
	
	Let $X$ be a left-invariant differential operator. We  denote with $X^R$ the right-invariant differential operator which induces the same point distribution as $X$ at $e$. In other words, $(X f)(e)=(X^R f)(e)$ for every $f\in C^\infty (G)$. We denote with $X^+$ the transpose of $X$ in the enveloping algebra $U(G)$. In other words, the mapping $X\mapsto X^+$ is the unique anti-automorphism of $U(G)$ (that is, such that $(XY)^+=Y^+ X^+$) which extends the automorphism $X\mapsto -X$ of $\gf$. 
	
	We denote with $X^\dag$ the formal transpose of $X$ (with respect to $\beta$), that is, the unique left-invariant differential operator such that
	\[
	\int_G (X f) g\,\dd \beta=\int_G f X^\dag g\,\dd \beta
	\]
	for every $f,g\in C^\infty_c(G)$. We denote with $X^*$ the formal adjoint of $X$, that is, $\overline X^\dag$. We define the formal transpose and the formal adjoint of right-invariant differential operators in a similar way. If $u$ is a distribution, we then define $X u$ so that
	\[
	\langle X u,\phi\rangle =\langle u, X^\dag \phi\rangle
	\]
	for every $\phi\in C^\infty_c(G)$. In this way, if $f\in C^\infty(G)$, then $(X f)\cdot \beta=X(f\cdot \beta)$.

	In order to simplify the notation, we write $X^{R\dag}$ instead of $(X^R)^\dag$, etc.
\end{deff}

Notice that $X f$ \emph{depends} on $\beta$ if $f$ is a distribution, whereas $X^\dag f$ does not. On the contrary, $X^\dag f$ \emph{depends} on $\beta$ if $f\in C^\infty(G)$, whereas $X f$ does not.  This unfavorable dichotomy disappears if $\beta$ is right-invariant. In the following proposition we collect some elementary facts; cf.~\cite[Proposition 2.2]{BCP} for a proof.

\begin{prop}\label{prop:9}
	The following hold:
	\begin{enumerate}
		\item[\textnormal{(1)}] $X^\dag f= \Delta_R^{-1} X^+(\Delta_R f)$ and $X^{R\dag} f=\Delta_L^{-1} X^{+R}(\Delta_L f)$ for every $X\in U(G)$ and for every $f\in C^\infty(G)$;
		
		\item[\textnormal{(2)}] $X\delta_e= \Delta_R \Rc(X^\dag \delta_e)$  and $X^R\delta_e=  \Delta_L \Rc (X^\dag \delta_e)$   for every $X\in U(G)$;
		
		\item[\textnormal{(3)}] $X^+ f=(X^R \check f) \check{\,}$ for every $X\in U(G)$ and for every $f\in C^\infty(G)$; 
		\item[\textnormal{(4)}] $X\delta_e=X^{\dag R\dag}\delta_e$ for every $X\in U(G)$;
		
		\item[\textnormal{(5)}] $X\chi= (X\chi)(e)\chi$ for every $X\in U(G)$ and for every character $\chi\colon G\to \C\setminus \Set{0}$.
	\end{enumerate}
\end{prop}

Notice that, by (4), 
\[
\begin{split}
	(Xf)*g=(f*X\delta_e)*g=f*(X\delta_e*g)=f*(X^{\dag R \dag} \delta_e*g)= f*(X^{\dag R \dag} g)
\end{split}
\]
under some reasonable conditions on $f$ and $g$ (which are needed to grant associativity of convolution).
Notice that $X^{\dag R \dag}-X^R$ is a differential operator whose order is strictly less than the order of $X$ (and whose degree is strictly less than the degree of $X$, with the terminology of the following subsection).

\subsection{Filtrations and Weighted Subcoercive Operators}

Throughout the paper, $(\gf_\lambda)_{\lambda\Meg 0}$ will denote an increasing filtration of $\gf$ (that is, $[\gf_\lambda, \gf_\mi]\subseteq \gf_{\lambda+\mi}$ for every $\lambda, \mi\Meg 0$) such that $\gf_\lambda=0$ for every $\lambda<1$, $\bigcup_{\lambda\Meg 0} \gf_\lambda=\gf$, and $\bigcap_{\mi>\lambda} \gf_\mi=\gf_\lambda$ for every $\lambda\Meg 0$.\footnote{We consider the full range $\lambda\Meg 0$  for notational convenience, in analogy with the filtration $(U_\lambda)$, for which $U_\lambda\neq \Set{0}$ for every $\lambda\Meg 0$.}

For every $X\in \gf$, we define  $\deg X\coloneqq \min\Set{\lambda\Meg0\colon X\in \gf_\lambda}$ and we call $\deg X$ the degree of $X$.
Define, for every $\lambda>0$, $\gf_{\lambda^-}\coloneqq \bigcup_{\mi<\lambda} \gf_\mi$, $\gf_{*,\lambda}\coloneqq\gf_\lambda/\gf_{\lambda^-} $, and 
\[
\gf_*\coloneqq \bigoplus_{\lambda>0} \gf_{*,\lambda}.
\]
Define a Lie algebra structure on $\gf_*$ as follows: if $X=\sum_{\lambda>0} (X_\lambda+ \gf_{\lambda^-})$ and $Y= \sum_{\lambda>0} (Y_\lambda+ \gf_{\lambda^-})$ for some $(X_\lambda),(Y_\lambda)\in \prod_{\lambda>0} \gf_\lambda$, then
\[
[X,Y]\coloneqq \sum_{\lambda,\mi>0}\left(  [X_{\lambda},Y_{\mi}]+\gf_{(\lambda+\mi)^-}\right) .
\]
It is easily seen that $(\gf_{*,\lambda})_{\lambda>0}$ is a graduation of type $((0,+\infty),+)$ of $\gf_*$, that is, $[\gf_{*,\lambda}, \gf_{*,\mi}]\subseteq \gf_{*,\lambda+\mi}$ for every $\lambda,\mi>0$. One may then endow $\gf_*$ with the automorphic dilations $(\delta_r)_{r>0}$ defined so that $\delta_r(X)=r^\lambda X$ for every $X\in \gf_{*,\lambda}$ and for every $\lambda>0$. Thus, $\gf_*$ is the Lie algebra of some homogeneous group $G_*$, with homogeneous dimension $Q_*\coloneqq \sum_{\lambda>0} \dim \gf_{*,\lambda}$.
\emph{We shall always assume, in the sequel, that $\Lambda\coloneqq \Set{\lambda \Meg 1\colon \gf_{*,\lambda}\neq 0}$ is contained in a one-dimensional $\Q$-vector space. In other words, setting $\dd_0\coloneqq \min \Lambda$, we assume that $\Lambda \subseteq \Q \dd_0$. This assumption ensures that (positive) Rockland operators on $G_*$   exist (and is essentially necessary).}

We observe explicitly that we required $\gf_\lambda=\Set{0}$ for $\lambda<1$ in order for the control modulus $\abs{\,\cdot\,}_*$ (cf.~Definition~\ref{def:3} below) to induce a left-invariant \emph{distance} on $G$ (rather than a quasi-distance).   We preferred to avoid imposing that $\gf_1\neq \Set{0}$ in order to keep a natural comparison with graded (or, more generally, homogeneous) groups: if $\gf$ has a graduation $(\tilde \gf_j)_{j\in \Z_+^*}$ (with integer degrees), then it is natural to set $\gf_\lambda=\bigoplus_{j=1}^{[\lambda]} \tilde \gf_j$ for every $\lambda \Meg 0$, but there is no guarantee that $\tilde \gf_1$ should be non-trivial.

We now extend this filtration to the  enveloping algebra $U(G)$.	
For every $\lambda\Meg 0$, define $U_\lambda$ as the vector space generated by the products of the form $X_1\cdots X_k$, for $k\Meg 0$, $X_1,\dots, X_k\in \gf$ and $\deg X_1+\cdots+ \deg X_k\meg \lambda$.\footnote{Thus, the identity operator, corresponding to the case $k=0$, belongs to all $U_\lambda$.} 
Then, $(U_\lambda)$ is an increasing filtration of $U(G)$, that is, $U_\lambda, U_\mi\subseteq U_\lambda U_\mi\subseteq U_{\lambda+\mi}$ for every $\lambda, \mi\Meg 0$.
For every $X\in U(G)$, we define $\deg X\coloneqq \min\Set{\lambda\Meg 0\colon X\in U_\lambda}$. Notice that $\gf\cap U_\lambda=\gf_\lambda$ for every $\lambda\Meg 0$ as a consequence of Proposition~\ref{prop:8} below (and its proof), so that this definition is consistent with the previous one.

Define $U_{\lambda^-}\coloneqq \bigcup_{\mi<\lambda} U_\mi$ for $\lambda>0$, $U_{0^-}\coloneqq\Set{0}$, $U_{*,\lambda}\coloneqq U_\lambda/U_{\lambda^-}$ for every $\lambda\Meg 0$, and 
\[
U_*\coloneqq \bigoplus_{\lambda\Meg 0} U_{*,\lambda} .
\] 
We define an algebra structure on $U_*$ as follows:  if $X=\sum_{\lambda\Meg 0} (X_\lambda+ U_{\lambda^-})$ and $Y= \sum_{\lambda\Meg 0} (Y_\lambda+ U_{\lambda^-})$ for some $(X_\lambda),(Y_\lambda)\in \prod_{\lambda\Meg 0} U_\lambda$, then
\[
XY\coloneqq \sum_{\lambda,\mi\Meg 0}\left(  X_{\lambda}Y_{\mi}+U_{(\lambda+\mi)^-}\right) .
\]
It is easily seen that $U_*$ becomes a graded algebra of type $([0,+\infty),+)$ with this structure.

\begin{prop}\label{prop:8}
	The canonical inclusions $\gf_\lambda \subseteq U_\lambda$, $\lambda> 0$, induce a linear mapping $\pi \colon \gf_*\to U_*$. The canonical extension $U(\pi)\colon U(G_*)\to U_*$ of $\pi$ is an  isomorphism of graded algebras.
\end{prop}
Cf.~\cite[Proposition 4.2]{BCP}
From now on, we shall identify $U_*$ and $U(G_*)$ by means of $U(\pi)$.

\begin{deff}
	We say that a family $(X_j)_{j\in J}$ of elements of $\gf$ is a minimal basis if, setting $Y_j\coloneqq X_j+ \gf_{(\deg X_j)^-}$, the family $(Y_j)$ induces a (homogeneous) basis of $\gf_*/[\gf_*,\gf_*]$.
	We denote with $\dd$ the least common multiple of the degrees of the $X_j$, $j\in J$.\footnote{Notice that $\dd$ is well defined since by our assumption the $d_j$ are all integer multiples of some element of $(0,+\infty)$.}
	
	We say that $\Lc $ is weighted subcoercive if $ \Lc+\Lc^*+U_{\grado^-}$ is a positive Rockland operator on $G_*$.
\end{deff}

If $(X_j)$ is a minimal basis of $\gf$, then $(X_j)$ generates $\gf$ as a Lie algebra, and also generates the filtrations $(\gf_\lambda)$ and $U_\lambda$ (cf.~\cite[Proposition 4.4]{BCP}). In other words, $\gf_\lambda$  is the vector space generated by the vector fields of the form $\ad(X_{j_1})\cdots \ad(X_{j_{k-1}})X_{j_k}$, where $k\Meg 1$, $j_1,\dots, j_k\in J$, and $\deg(X_{j_1})+\cdots +\deg(X_{j_k})\meg \lambda$. Analogously, $U_\lambda$ is the vector space generated by the differential operators for the form $X_{j_1}\cdots X_{j_k}$, where $k\Meg 0$, $j_1,\dots, j_k\in J$, and $\deg(X_{j_1})+\cdots +\deg(X_{j_k})\meg \lambda$.
In particular, $(X_j)$ is a `reduced weighted algebraic basis' of $\gf$, in the terminology of~\cite{ElstRobinson}.

Observe that $\dd$ does \emph{not} depend on the choice of $(X_j)$, since it is the least common multiple of the degrees of the non-zero elements of $\gf_*/[\gf_*,\gf_*]$.

\begin{prop}\label{prop:4}
	Let $(X_j)_{j\in J}$ be a minimal basis of $\gf$. Then, $\Lc\coloneqq \sum_{j\in J} (X_j^{\dd/d_j})^\dag X^{\dd/d_j}_j$, where $d_j=\deg X_j$ for every $j\in J$, is a real,  positive and formally self-adjoint  weighted subcoercive operator.
\end{prop}

Here, by `positive' we mean that $\int \Lc f \overline f \,\dd \beta\Meg 0$ for every $f\in C^\infty_c(G)$.

Notice that, if $\gf_1$ is a vector subspace of $\gf$ which generates $\gf$ as a Lie algebra, and if $(\gf_\lambda)$ is the filtration generated by $\gf_1$ (that is, $\gf_\lambda$ is the vector space generated by the elements of $\gf_1$ -- if $\lambda\Meg 1$ -- and their commutators up to order $[\lambda]$), then any basis of $\gf_1$ is a minimal basis (and conversely), and the operator $\Lc$ defined above is a sub-Laplacian (with drift unless $\beta$ is right-invariant). This observation provides the link with the theory developed in~\cite{BPV,BPV2,BPV3}.

We shall now define a control modulus on $G$. 
\begin{deff}\label{def:3}
	Given an absolutely continuous curve $\gamma\colon [0,1]\to G$, we  define the content of $\gamma$ as the greatest lower bound of the $\eps>0$ such that 	
	\[
	\abs{P_\lambda \dd L_{\gamma(t)}^{-1}\gamma'(t)} \meg \min(\eps, \eps^{\lambda}) 
	\]
	for almost every $t\in [0,1]$ and for every $\lambda>0$, where $P_\lambda$ is the orthogonal projector of $\gf$ onto $\gf_\lambda \ominus \gf_{\lambda^-}$ (this is non-zero only for finitely many $\lambda>0$) and $L_{\gamma(t)}$ is the left translation by $\gamma(t)$. 
	
	Given $x\in G$, we shall define $\abs{x}_*$  as the greatest lower bound of the contents of the absolutely continuous curves $\gamma\colon [0,1]\to G$ such that $\gamma(0)=e$ and $\gamma(1)=x$.
	
	We  endow $G$ with the left-invariant distance $d(x,y)\coloneqq \abs{y^{-1}x}_*$.
\end{deff} 

Choosing a different scalar product on $\gf$   gives rise to bi-Lipschitz equivalent control moduli. Equivalence at infinity follows from the fact that all these control moduli are `connected moduli,'\footnote{In fact, every $x\in G$ may be written as $x_1\cdots x_k$, where $\abs{x_j}_*\meg 1$ for $j=1,\dots, k$, and $k\meg \abs{x}_*+1$.}  whereas equivalence near $e$ follows from~\cite[Corollary 6.5]{ElstRobinson}.  
Notice that here we are not requiring $\gamma$ to be `horizontal.' One may also require $\gamma$ to be horizontal with respect to some fixed weighted algebraic basis compatible with the filtration $(\gf_\lambda)$, in the terminology of~\cite{ElstRobinson}, and still get an equivalent control modulus. This would provide a better mean value theorem for the corresponding `horizontal gradient,' but we shall not need this kind of more precise estimates.

It is known that 
\[
\beta(B(e,r))\asymp r^{Q_*}, \qquad r\to 0^+,
\]
while
\[
\beta(B(e,r))\meg \ee^{C r}
\]
for some constant $C>0$ and for every $r\Meg 1$ (cf., e.g.,~\cite[Sect.\ 2.3]{Martini}).

\begin{deff}
	From now on, we shall fix a  weighted subcoercive operator $\Lc$ with degree $\grado$. We shall denote with $(h_t)_{t>0}$ the corresponding heat kernel. In other words, $\ee^{-t\Lc}f=f*h_t$ for every $f\in L^2(\beta)$, where $(\ee^{-t\Lc})_{t>0}$ is the semigroup generated by the closure of $\Lc$, with initial domain $C^\infty_c(G)$, in $L^2(\beta)$ (cf.~Theorem~\ref{teo:7} below).
	
	For every $\omega\in \R$, we shall set $\Lc_\omega \coloneqq \Lc+ \omega I$.
\end{deff}

Cf.~\cite[Theorems 4.8 and 5.4]{BCP} for a proof of the following result.

\begin{teo}\label{teo:7}
	There is $\omega\in \R$ such that the following hold:
	\begin{enumerate} 
		\item[\textnormal{(1)}] the closure of $\Lc$, with initial domain $C^\infty_c(G)$,  generates a semigroup of operators of $L^2(\beta)$; if $\Lc=\Lc^*$, then $\Lc$ is essentially self-adjoint on $C^\infty_c(G)$;
		
		\item[\textnormal{(2)}] for every $\lambda,\lambda'\Meg 0$ there are $b,C>0$ such that  
		\[
		\abs{X Y^R h_t(x)}\meg C \abs{X}\abs{Y} t^{-(Q_* + \lambda+\lambda')/\grado} \ee^{\omega t} \ee^{- b (\abs{x}_*^{\grado}/t)^{1/(\grado-1)}}
		\]
		for every $X\in U_\lambda$, for every $Y\in U_{\lambda'}$, and for every $x\in G$;
	\end{enumerate}
\end{teo}

\subsection{Gaussian Estimates}

\begin{deff}
	In order to simplify the notation, given a measure space $(X,\mi)$, $p\in (0,\infty]$, and a $\mi$-measurable function $f$ on $X$, we shall also write $\norm{f(x)}_{L^p_x(\mi)}$ instead of $\norm{f}_{L^p(\mi)}$ (we are therefore allowing the possibility $\norm{f}_{L^p(\mi)}=+\infty$). This will be particularly useful in the presence of nested $L^p$-norms.
\end{deff}

\begin{deff}
	For every $b,t>0$ and for every $d>1$, define 
	\[
	p_{b,t,d}\colon x\mapsto t^{-Q_*/d} \ee^{-b   (\abs{x}_*^{d}/t)^{1/(d-1)}}
	\]
	and
	\[
	T_{b,t,d}f\coloneqq \abs{f}*p_{b,t,d}
	\]
	for every $\beta$-measurable function $f$ on $G$. 
	We shall simply write $T_{b,t}$ and $p_{b,t}$ instead of $T_{b,t,\grado}$ and $p_{b,t,\grado}$, respectively.
\end{deff}

We now collect some elementary facts. Cf.~\cite[Lemmas 5.5, 5.6, 5.7, 5.8]{BCP} for the relative proofs.

\begin{lem}\label{lem:18}\label{lem:32}\label{lem:3}\label{cor:3}\label{lem:4}\label{lem:4bis}
	Take $b, b'>0$, $\rho\Meg 0$, and $d'\Meg d>1$. Then, the following hold:
	\begin{itemize}
		\item if $b\Meg b'$, then $		p_{b,t,d}\meg \ee^{b'} p_{b',t^{d'/d},d'}$	for every $t>0$;
		
		\item for every $x,y\in G$ and for every $\beta$-measurable function $f$ on $G$,
		\[
		\ee^{-2^{1/(d-1)}b(\abs{y^{-1}x}^d/t)^{1/(d-1)}}(T_{2^{1/(d-1)}b,t,d} f)(y)\meg (T_{b,t,d} f)(x)\meg \ee^{2^{1/(d-1)}b(\abs{y^{-1}x}^d/t)^{1/(d-1)}} (T_{2^{-1/(d-1)}b,t,d} f)(y);
		\]
		
		\item there is a constant $C>0$ such that
		\[
		\norm*{ p_{b,t,d}(x)\ee^{\rho \abs{x}_*}}_{L^p_x(\beta)}\meg C t^{-Q_*/(d p')}\ee^{C t}
		\]
		for every $t>0$ and for every $p\in [1,\infty]$;
		
		\item the mapping $(0,+\infty)\ni t \mapsto \ee^{\rho\abs{\,\cdot\,}_*}p_{t,b,d}\in L^p( \beta)$ is continuous for every $p\in [1,\infty]$;
		
		\item for every $b''\in (0,2^{-1/(d-1)}\min(b,b'))$  there is a constant $C>0$ such that
		\[
		\int_G p_{b,t,d}(x y^{-1}) p_{b',t',d}(y) \ee^{\rho\abs{y}_*}\,\dd \beta(y) \meg C \ee^{C (t+t')} p_{ b'',t+t',d}(x)
		\]
		for every $x\in G$ and for every $t,t'>0$;
		
		\item for every $b''> 2^{1/(d-1)}\max(b,b')$ there is a constant $C>0$ such that
		\[
		\int_G p_{b,t,d}(x y^{-1}) p_{b',t',d}(y) \ee^{-\rho\abs{y}_*}\,\dd \beta(y) \Meg C \ee^{-C(t+t')} p_{ b'',t+t',d}(x)
		\]
		for every $x\in G$ and for every $t,t'>0$.
	\end{itemize}
\end{lem}

%
%
%

%
%
\begin{deff}
	Take $ b>0$, $\omega\in \R$, and $d>1$. Then, we define  
	\[
	T^{\omega}_{b,*,d} f(x) \coloneqq  \sup_{t\in (0,\kappa]}  \ee^{-t\omega}(T_{b,t,d} f)(x)
	\]
	for every $\beta$-measurable function $f$ on $G$ and for every $x\in G$. 
	
	We shall generally omit $d$ if it is $\grado$.
\end{deff}

%
%
%
%
%

\subsection{The `Schwartz Space'}

\begin{deff}\label{def:1}
	We define $\Sc(G)$ as the space of $f\in C^\infty(G)$ such that the seminorms $\norm{\ee^{c \abs{\,\cdot\,}_*} X^R f}_{L^1(\beta)}$, $c>0$, $X\in U(G)$, are finite, endowed with the corresponding topology.
	
	We denote with $\Sc'(G)$ the dual of $\Sc(G)$, endowed with the topology of uniform convergence on the bounded subsets of $\Sc(G)$.
\end{deff}

Observe that, by the arbitrariness of $c$, one may replace $\beta$ with $\beta_L$ in Definition~\ref{def:1}. Consequently, the results in~\cite{Schweitzer} are applicable in this context.

Cf.~\cite[Theorem 4.16, Corollary 4.17, and Lemma 5.13]{BCP} for a proof of the following results.

\begin{teo}\label{teo:8}
	The following hold:
	\begin{enumerate}
		\item[\textnormal{(1)}] $\Sc(G)$ is a nuclear Fréchet space;
		
		\item[\textnormal{(2)}] $\Sc(G)$ is reflexive;
		
		\item[\textnormal{(3)}] the bounded subsets of $\Sc(G)$ are relatively compact;
		
		\item[\textnormal{(4)}] $\Sc(G)$ is a Fréchet $*$-algebra under convolution and under pointwise multiplication;
		
		\item[\textnormal{(5)}] for every $p\in [1,\infty]$, $\Sc(G)$ is the space of $f\in C^\infty(G)$ such that the seminorms  $\norm{\ee^{c \abs{\,\cdot\,}_*} X Y^R f}_{L^p(\beta)}$, $c>0$, $X,Y\in U(G)$ (resp.\ $\norm{\ee^{c \abs{\,\cdot\,}_*} X   f}_{L^p(\beta)}$, $c>0$, $X\in U(G)$; $\norm{\ee^{c \abs{\,\cdot\,}_*} X^R f}_{L^p(\beta)}$, $c>0$, $X\in U(G)$), are finite, and has the corresponding topology;
		
		\item[\textnormal{(6)}] 	$\Sc'(G)$ is a complete, reflexive, bornological, and nuclear space.  The bounded subsets of $\Sc'(G)$ are relatively compact.
	\end{enumerate} 
\end{teo}

\subsection{Lattices}

\begin{deff}
	Take $\delta>0$ and $R\Meg 2$. We say that a family $(x_j)_{j\in J}$ of elements of $G$ is a $(\delta,R)$-lattice if the $B(x_j,\delta)$ are pairwise disjoint, while the $B(x_j,R\delta)$ cover $G$.
\end{deff}

It is clear that any maximal $2\delta$-separated family of elements of $G$ is a $(\delta,2)$-lattice, so that existence of $(\delta,R)$-lattices is ensured. It is nonetheless convenient, for technical reasons, to allow for `coarser' lattices (i.e., for the case $R>2$).  Cf.~\cite[Lemma 4.3]{Calzi2} for a proof of the following result.

\begin{lem}\label{lem:50}
	Take $\delta_0>0$ and $R_0\Meg 2$. Then, there is $N\in\N$ such that
	\[
	1\meg \sum_{j\in J} \chi_{B(x_j,R\delta)}\meg N
	\]
	for every $(\delta,R)$-lattice $(x_j)_{j\in J}$ on $G$ with $\delta\meg \delta_0$ and $2\meg R\meg R_0$. In particular, it is possible to find a partition $J_1,\dots, J_N$ of $J$ such that $d(x_j,x_{j'})\Meg  R\delta$ for every two distinct $j,j'\in J_h$, and for every $h=1,\dots, N$.
\end{lem}

Notice that every $(\delta,R)$-lattice is also a $(\delta',R')$-lattice for every $\delta'\meg \delta$ and for every $R'\Meg R\delta/\delta'$, so that the above assertion gives information also on the possible overlaps of balls centred at the $x_j$ and of arbitrarily large (but fixed) radii.

\subsection{Supplementary Lemmas}

\begin{deff}
	We denote with $\Mc_\Car$ (`Carleson measures') the space of positive Radon measures $\mi$ on $(0,+\infty)$ with bounded support such that there is a constant $C>0$ such that
	\[
	\mi((r,2r])\meg C
	\]
	for every $r>0$.
	
	We denote with $\Mc_\RC$  (`reverse Carleson measures') the space of positive Radon measures $\mi$ on $(0,+\infty)$ such that there are constants $\eps,C\in (0,1)$   such that
	\[
	\mi((\eps r, r])\Meg C
	\]
	for every $r\in (0,\eps]$.
	
	We define $\Mc_\Samp\coloneqq \Mc_\Car\cap \Mc_\RC$  (`sampling measures').
	
	Finally, for every $\kappa>0$ we define $\mi_\kappa$ as the Radon measure on $(0,+\infty)$  supported in $(0,\kappa]$ and such that $\dd\mi_\kappa(t)=\frac{\dd t}{t}$ on $(0,\kappa]$.
\end{deff}
%
%
%

Cf.~\cite[Lemmas 6.4 and 6.6]{BCP} for a proof of the following results.

\begin{lem}\label{lem:25}
	Take $\eta,\gamma>0$,  $\mi,\nu\in \Mc_\Car$. Then, there is a constant $C>0$ such that
	\[
	\norm*{\int_0^\infty \frac{s^\gamma t^{\eta}}{(s+t)^{\eta+\gamma}} \abs{f(t)}\,\dd \nu(t)  }_{L^q_s(\mi)} \meg C \norm{f}_{L^q(\nu)}
	\]
	for every  $q\in [1,\infty]$ and for every   $\nu$-measurable function $f$.
\end{lem}

\begin{lem}\label{lem:25bis}
	Take $\eta,\gamma>0$ and  let $\mi$ be a Haar measure on $(0,+\infty)$. Take $\nu\in \Mc_\Car$, and assume that either $\eta\Meg 1$ or $\nu\in L^\infty(\mi)\cdot \mi$.	
	Then, there is a constant $C>0$ such that
	\[
	\norm*{\int_0^\infty t^\eta s^\gamma \abs{f(s+t)}\,\dd \mi(t)  }_{L^q_s(\nu)} \meg C \norm{ t^{\gamma+\eta}f(t)}_{L^q_t(\mi)}
	\]
	for every  $q\in [1,\infty]$ and for every  $\mi$-measurable function $f$.
\end{lem}

\subsection{Besov and Triebel-Lizorkin Spaces}

\begin{deff}\label{def:2}
	Define, for every $k\in \N$, for every $t>0$, and for every $f\in \Sc'(G)$,
	\[
	W_t^{(k)} f\coloneqq (t \Lc)^k \ee^{-t \Lc} f.
	\]
	In addition, define,  for every $\alpha\in\R$ and for every $\eps\in (0,1]$,
	\[
	W_{t,\eps}^{(\alpha),*} f\coloneqq t^\alpha \max_{s\in [\eps t,t/\eps ]} \max_{\substack{\deg(X)+\deg(Y)\meg \alpha \grado \\ \abs{X} ,\abs{Y}\meg 1}} \abs{X \ee^{-s \Lc} Y f }.
	\]
\end{deff}

\begin{deff}
	For every $\alpha \in \R$, for every $m\in\N$ with $m>\alpha/\grado$, for every $\mi\in \Mc_\Samp$ and for every $p,q\in [1,\infty]$,  define
	\[
	\Bs^{p,q}_{\alpha,m,\mi}(f)\coloneqq \norm{   t^{-\alpha/\grado} \norm{ W_t^{(m)} f }_{L^p(\beta)}  }_{L^q_t(\mi)},
	\]
	and, if $p\in (1,\infty)$,
	\[
	\Fs^{p,q}_{\alpha,m,\mi}(f)\coloneqq \norm*{ \norm{ t^{-\alpha/\grado}  W_t^{(m)} f(x)  }_{L^q_t(\mi_1)}}_{L^p_x(\beta)} ,
	\]
	Then, define, for every $f\in \Sc'(G)$,
	\[
	\norm{f}_{B^{p,q}_\alpha(\beta)}\coloneqq \norm{\ee^{-\Lc} f}_{L^p(\beta)}+ \Bs^{p,q}_{\alpha,([\alpha/\grado]+1)_+,\mi_1}(f),
	\]
	and, if $p\in (1,\infty)$,
	\[
	\norm{f}_{F^{p,q}_\alpha(\beta)}\coloneqq \norm{\ee^{-\Lc} f}_{L^p(\beta)}+ \Fs^{p,q}_{\alpha,([\alpha/\grado]+1)_+,\mi_1}(f).
	\] 
	We define $B^{p,q}_\alpha(\beta)$ and $F^{p,q}_\alpha(\beta)$ accordingly.
\end{deff}

By~\cite[Proposition 6.9]{BCP}, the norms $\norm{\ee^{-\Lc} \,\cdot\,}_{L^p(\beta)}+\Bs^{p,q}_{\alpha,m,\mi}$ are all equivalent on $\Sc'(G)$, as well as the norms $\norm{\ee^{-\Lc} \,\cdot\,}_{L^p(\beta)}+\Fs^{p,q}_{\alpha,m,\mi}$. Hence, they may be used to define the spaces $B^{p,q}_\alpha(\beta)$ and $F^{p,q}_\alpha(\beta)$, respectively.

Let us also recall the following simple result (cf.~\cite[Proposition 7.4]{BCP}) which will be often useful to reduce to the case in which $\beta$ is left invariant.

\begin{prop}\label{prop:5}
	Suppose $\alpha\in \R$ and $p,q\in [1,\infty]$. Then 
	\begin{enumerate}
		\item[\textnormal{(i)}]  $B^{p,q}_\alpha(\beta)= \Delta_L^{-1/p} B^{p,q}_\alpha(\beta_L)$, and
		\item[\textnormal{(ii)}]  $F^{p,q}_\alpha(\beta)= \Delta_L^{-1/p} F^{p,q}_\alpha(\beta_L)$  if $p\in (1,\infty)$.
	\end{enumerate}
\end{prop}

\section{The Spaces $\Cc^{q}(\mi,a)$}\label{sec:3}

In order to define the spaces $F^{\infty,q}_\alpha(\beta)$, we (essentially) follow the definition introduced in~\cite{FrazierJawerth}, partially following the approach of~\cite{Rychkov}. Since in this more general context we do not have access to a suitable `discretization,' we shall need to make some changes and weaken some statements. We consider the case of general $q\in (0,\infty]$ for future reference.

\begin{deff}
	Take $q\in(0,\infty]$, $a>0$, and a positive measure $\mi$ on $(0,+\infty)$ with bounded support. Set $c\coloneqq \max \supp \mi$. Then, we denote with $\Cc^q(\mi,a)$ the space of $(\mi\otimes \beta)$-measurable functions $f\colon (0,+\infty)\times G\to \C$ such that 
	\[
	\sup_{(t,x)\in (0,c]\times G}\norm*{t^{-aQ_*/q} \chi_{(0,t]\times B(x,t^a)} f }_{L^q(\mi\otimes \beta_L)} 
	\]
	is finite, 	endowed with the corresponding topology. We shall simply write $\Cc^q(\mi)$ instead of $\Cc^q(\mi,1)$.
\end{deff}

Clearly, $\Cc^\infty(\mi,a)=L^\infty(\mi\otimes\beta)$ for every $a>0$.

We observe explicitly that the parameter $a>0$ does not appear in~\cite{FrazierJawerth,Rychkov}. We introduced this parameter since it should be naturally chosen as $1$ in an abstract treatment, but when applied to functions of the form $(t,x)\mapsto (W^{(m)}_t g)(x)$ it is more naturally chosen as $1/\grado$ (unless one chooses to deal with the somewhat less natural function $(t,x)\mapsto (W^{(m)}_{t^{\grado}} g)(x)$). When $\mi$ is discrete (and appropriately chosen), these issues do not occur.

In the next few lemmas, we discuss some basic properties of the spaces $\Cc^q(\mi,a)$.

\begin{lem}\label{lem:63}
	Take $q\in(0,\infty]$, $a>0$, $\kappa>1$, and a positive measure $\mi$ on $(0,+\infty)$ with bounded support. Then, there is a constant $C>0$ such that
	\[
	\sup_{(t,x)\in  (0,\kappa c]\times G}\norm*{t^{-aQ_*/q} \chi_{(0,t]\times B(x,\kappa t^a)} f }_{L^q(\mi\otimes \beta_L)}\meg C \norm{f}_{\Cc^q(\mi,a)}
	\]
	for every $(\mi\otimes \beta)$-measurable function $f\colon (0,+\infty)\times G\to \C$, where $c\coloneqq \max \supp \mi$.
\end{lem}

\begin{proof}
	Observe that  Lemma~\ref{lem:50} shows that there is $N\in\N$ such that, for every $t\in (0,\kappa c]$ and for every $(\min(t,c)^a/2,2)$-lattice $(y_{t,j})_{j\in \N}$, there is a subset $J_t$ of $\N$ such that $B(e,\kappa t^a)\subseteq \bigcup_{j\in J_t} B(y_{t,j},\min(t,c)^a)$ and $\card(J_t)\meg N$. Then,
	\[
	\begin{split} 
		\norm*{t^{-aQ_*/q} \chi_{(0,t]\times B(x,\kappa t^a)} f }_{L^q(\mi\otimes \beta_L)}^{\min(1,q)}&\meg \norm*{\min(t,c)^{-aQ_*/q} \chi_{(0,\min(t,c)]\times B(x,\kappa t^a)} f }_{L^q(\mi\otimes \beta_L)}^{\min(1,q)}\\
			&\meg   \sum_{j\in J_t}\norm*{\min(t,c)^{-aQ_*/q} \chi_{(0,\min(t,c)]\times B(x y_{t,j}, \min(t,c)^a)} f }_{L^q(\mi\otimes \beta_L)}^{\min(1,q)}\\
			&\meg  N \norm{f}_{\Cc^q(\mi,a)}^{\min(1,q)}
	\end{split}
	\]
	for every $x\in G$ and for every $t\in (0,\kappa c]$. The assertion follows.
\end{proof}

\begin{lem}\label{lem:61}
	Take $q\in(0,\infty]$, $a>0$, and a positive measure $\mi$ on $(0,+\infty)$ with bounded support. Take $\delta>0$ and let $\mi_\delta$ be the push-forward of $\mi$ under the mapping $t\mapsto t/\delta$. Then, the mapping
	\[
	T_\delta\colon \Cc^q(\mi,a)\ni f \mapsto [(t,x)\mapsto f(\delta t,x) ]\in \Cc^q(\mi_\delta,a)
	\]
	is an isomorphism.
\end{lem}

\begin{proof}
	We may assume that $\beta=\beta_L$; set $c\coloneqq \sup \supp \mi$, so that $c/\delta=\sup \supp \mi_\delta$.
	It will suffice to prove continuity for $\delta\neq 1$. If $\delta>1$, the assertion is clear, since
	\[
	\begin{split}
		\norm{T_\delta f}_{\Cc^q(\mi_\delta,a)}&=\sup_{(t,x)\in  (0,c/\delta]\times G}\norm*{t^{-aQ_*/q} \chi_{(0,t]\times B(x,t^a)}  (s,y) f(\delta s,y) }_{L^q_{(s,y)}(\mi_\delta \otimes \beta )}\\
			&=\sup_{(t,x)\in  (0,c/\delta]\times G}\norm*{t^{-aQ_*/q} \chi_{(0,\delta t ]\times B(x,t^a)}   f }_{L^q(\mi \otimes \beta )}\\
			&=\delta^{aQ_*/q} \sup_{(t,x)\in  (0,c]\times G}\norm*{t^{-aQ_*/q} \chi_{ (0,t]\times B(x,t^a/\delta^a)}   f }_{L^q(\mi \otimes \beta )}\\
			&\meg  \delta^{aQ_*/q}\norm{f}_{\Cc^q(\mi,a)}
	\end{split}
	\]
	for every $(\mi \otimes \beta)$-measurable function $f$. Next, assume that $\delta\in (0,1)$.  
	Then, arguing as before and using Lemma~\ref{lem:63}, we see that there is a constant $C>0$ such that
	\[
	\begin{split}
		\norm{T_\delta f}_{\Cc^q(\mi_\delta,a)}& = \delta^{aQ_*/q}\sup_{(t,x)\in  (0,c]\times G}\norm*{t^{-aQ_*/q} \chi_{ (0,t]\times B(x,t^a/\delta^a)}   f }_{L^q(\mi \otimes \beta )}\\ 
		&\meg  C\norm{f}_{\Cc^q(\mi,a)},
	\end{split}
	\]
	for every $(\mi \otimes \beta)$-measurable function $f$, whence the conclusion.
\end{proof}

The following result shows the relation between $\Cc^q(\mi)$ and $\Cc^q(\mi,a)$. In particular, it allows to reduce to a specific $a$ in several proofs, if desired.

\begin{lem}\label{lem:61bis}
	Take $q\in(0,\infty]$, $a,a'>0$, and a positive measure $\mi$ on $(0,+\infty)$ with bounded support. Let $\mi^{(a')}$ be the push-forward of $\mi$ under the mapping $t\mapsto t^{a'}$. Then, the mapping
	\[
	T_\delta\colon \Cc^q(\mi,a)\ni f \mapsto [(t,x)\mapsto f(t^{1/a'},x) ]\in \Cc^q(\mi^{(a')},a/a')
	\]
	is an isometric isomorphism.
\end{lem}

\begin{lem}\label{lem:2bis}
	There is $\omega\in \R$ such that the following hold.
	Take $ q\in [1,\infty]$, $a,b>0$, $c,d>1$,   a positive Radon measure $\mi$ on $(0,+\infty)$ with bounded support,  and a  function $\kappa\colon (0,+\infty)\to (0,+\infty)$ such that $\kappa(t)\in (0,c t^{a d}]\cup [1/c,+\infty)$ for every $t>0$. 
	Then, there is   $C>0$ such that, for every $(\mi \otimes \beta)$-measurable function $f\colon (0,+\infty) \times G\to [0,+\infty]$,
	\[
	\begin{split}
		\norm{\ee^{-\omega\kappa(t)} [T_{b,\kappa(t),d} f(\,\cdot\,,t)](x) }_{\Cc^q_{(t,x)}(\mi,a)} \meg C  \norm{  f }_{\Cc^q(\mi,a)} .
	\end{split}
	\]
\end{lem}

This result replaces~\cite[Lemma 5.11]{BCP} in most occurrences; the proof is inspired by that of~\cite[Lemma 5]{Rychkov}. It would  be desirable to have an analogue of the above result (possibly for $q>1$) with $\ee^{-\omega \kappa(s)} T_{b,\kappa(s),d}$ replaced by $T^{\omega}_{b,*,d}$. This would allow to provide a cleaner description of the space $\mathring F^{\infty,q}_\alpha(\beta)$ (defined below), but we have not been able to find a proof.

\begin{proof} 
	We may assume that $\beta=\beta_L$, and also that $a=1$, thanks to Lemma~\ref{lem:61bis}.		
	We define $(T f)(t,x)\coloneqq \ee^{-\omega \kappa(t)}T_{b,\kappa(t),d}(f(t,\,\cdot\,))(x)$.
	The measurability of $Tf$ may be established as in the proof of~\cite[Lemma 5.11]{BCP}. 
	Take $\omega>0$ so that $T^{\omega}_{b,*,d}$ is of strong type $(r,r)$ for every $r\in (1,\infty]$  (cf.~\cite[Lemma 5.10]{BCP}) and such that 
	\[
	C_1\coloneqq \sup_{t>0} \sup_{r\in [1,\infty]}\ee^{-\omega t}\norm{ \Delta_R^{-1/r} p_{b,t,d}}_{L^1(\beta)} 
	\]
	is finite  (cf.~Lemma~\ref{lem:3}).
	Observe that the case $q=\infty$ is trivial, since $\Cc^\infty(\mi,1)=L^\infty(\mi\otimes \beta)$.
	We may therefore assume that $q\in [1,\infty)$.
	
	Set $c'\coloneqq \max \supp \mi$ and take, for every $t\in (0,c']$, a $(t/2,2)$-lattice $(y_{t,j})_{j\in \N}$ on $G$.
	Then, by Lemma~\ref{lem:50}, there is a constant $N\in\N$ such that   $\sum_j \chi_{B(y_{t,j},5t)}\meg N$ for every $t\in (0,c']$. 
	If $q>1$, then by the previous remarks and Lemma~\ref{lem:63} there is a constant $C_2>1$ such that
	\[
	\begin{split}
		&\norm*{t^{-Q_*/q}\chi_{(0,t]\times B(x,t)}  T(\chi_{(0,+\infty)\times B(x,4 t)} f ) }_{L^q(\mi \otimes \beta)}\\
		&\qquad\meg\norm*{t^{-Q_*/q}\chi_{ (0,t]}( s)  T^{\omega}_{b,*,d}(\chi_{  B(x,4 t)} f(s,\,\cdot\,) )(y)}_{L^q_{(y,s)}(\mi \otimes \beta)}\\
		&\qquad\meg C_2 \norm*{t^{-Q_*/q}\chi_{(0,t]\times B(x,4 t)}  f }_{L^q (\mi \otimes \beta)}\\ 
		&\qquad\meg C_2^2  \norm{f}_{\Cc^q(\mi)}
	\end{split}
	\]
	for every $x\in G$ and for every $t\in (0,c']$.
	If, otherwise, $q=1$, then by the previous remarks and Lemma~\ref{lem:63} we see that there is a constant $C_3>0$ such that, for every $x\in G$ and for every $t\in (0,c']$,
	\[
	\begin{split}
		&\norm*{t^{-Q_*}\chi_{(0,t]\times B(x,t)}   T(\chi_{ (0,+\infty)\times B(x,4 t)} f ) }_{L^1 (\mi \otimes \beta)}\\
		&\qquad= t^{-Q_*}\sup_{g\in \Gc_t}  \int_{ (0,t]} \ee^{-\omega \kappa(s)} \langle (\chi_{B(x,4 t)} {f(s,\,\cdot\,)} )*p_{b,\kappa(s),d}\vert g(s,\,\cdot\,) \rangle \,\dd \mi(s) \\
		&\qquad\meg t^{-Q_*}\sup_{g\in \Gc_t} \int_{ (0,t]} \ee^{-\omega \kappa(s)}  \langle \chi_{B(x,4 t)} {f(s,\,\cdot\,)} \vert g(s,\,\cdot\,)* p_{b,\kappa(s),d} \rangle  \,\dd \mi(s)\\
		&\qquad \meg C_1 t^{-Q_*} \int_{ (0,t]}   \int_{B(x,4 t)} {f(s,y)}\,\dd \beta(y) \,\dd \mi(s)\\ 
		&\qquad \meg C_3\norm{f}_{\Cc^q(\mi)}
	\end{split}
	\]
	where $\Gc_t$ is the set of \emph{positive} $g\in L^\infty(\mi \otimes \beta)$ concentrated on $(0,t]\times B(x,t) $ and such that $\norm{g}_{L^\infty(\mi \otimes \beta)}\meg 1$.
	
	It then remains to estimate $\norm*{t^{-Q_*/q}\chi_{(0,t]\times B(x,t)}  T(\chi_{(0,+\infty)\times [G\setminus B(x,4 t)]} f ) }_{L^q(\mi \otimes \beta)}$ for $q\in [1,\infty)$.
	Observe first that Jensen's inequality shows that
	\[
	\begin{split}
		(T g)(s,y)^q&=  \left( \ee^{-\omega \kappa(s)} \int_G {g(s,yz^{-1})} p_{b,\kappa(s),d}(z)\,\dd\beta_R(z)\right) ^q\\
		&\meg C_1^{q/q'}\ee^{-\omega \kappa(s)} \int_G {g(s,yz^{-1})}^q p_{b,\kappa(s),d}(z)\,\dd\beta_R(z)\\
		&=C_1^{q/q'} T (g^q)(s,y)
	\end{split}
	\]
	for every $(\mi\otimes \beta)$-measurable function $g\colon (0,+\infty)\times G\to [0,+\infty]$ and for every $(s,y)\in (0,+\infty) \times G$. Applying the previous inequality to $g=\chi_{(0,t]\times [G\setminus B(x,4t)]}f $ and setting $J_t\coloneqq \Set{j\in \N\colon \abs{y_{t,j}}_*\Meg  3t}$, we then see that
	\[
	\begin{split}
		T(\chi_{ (0,t]\times [G\setminus B(x,4t)]} f )(s,y)^q&\meg C_1^{q/q'} T(\chi_{ (0,t]\times [G\setminus B(x,4t)]} {f }^q)(s,y)\\
		&\meg C_1^{q/q'}\sum_{j\in J_t} T (\chi_{ (0,t]\times B(x y_{t,j},t)} {f }^q)(s,y)
	\end{split}
	\]
	for every $(s,y)\in (0,t] \times G$, for every $x\in G$, and for every $t\in (0,c']$.
	Now, observe that, if   $y\in B(x,t)$, then  $\abs{z^{-1}y}_*\Meg \abs{y_{t,j}^{-1} x^{-1} y}_*-t\Meg \abs{y_{t,j}^{-1}}_*-2 t\Meg \abs{y_{t,j}}_*/3$ for every $z\in B(x y_{t,j},t)$ and for every $j\in J_t$, so that, setting $b'\coloneqq 3^{-d/(d-1)} b $,
	\[
	\begin{split}
		&T (\chi_{ (0,t]\times B(x y_{t,j},t)} {f }^q)(s,y) =  \ee^{-\omega \kappa(s)} \kappa(s)^{-Q_*/d} \int_{B(x y_{t,j},t)}  {f(s,z)}^q \ee^{-b(\abs{z^{-1}y}_*^{d}/\kappa(s))^{1/(d-1)}}\,\dd \beta(z)\\ 
		&\qquad\meg   \sup_{t'\in (0, c t^d]\cup [1/c,+\infty)} t'^{-Q_*/d}\ee^{-\omega t'-b'(\abs{y_{t,j}}_*^d/t')^{1/(d-1)}}  \int_{B(x y_{t,j},t)}  {f(s,z)}^q  \,\dd \beta(z).
	\end{split}
	\]
	Consequently, 
	\[
	\begin{split}
		&\norm*{\chi_{(0,t]\times B(x,t)} T (\chi_{(0,t]\times [G\setminus B(x,4t)]} {f }^q)}^q_{L^q(\mi \otimes \beta)}\\
		&\qquad\meg C_1^{1/q'} \beta(B(e,t))\sum_{j\in J_t }\sup_{t'\in (0, c t^d]\cup [1/c,+\infty)} t'^{-Q_*/d}\ee^{-\omega t'-b'(\abs{y_{t,j}}_*^{d}/t')^{1/(d-1)}} \sup_{y\in G}\norm*{\chi_{(0,t]\times B(y,t)}  f}^q_{L^q(\mi \otimes \beta)} .
	\end{split}
	\]
	It will therefore suffice to show that both
	\[
	\sup_{t\in (0,c']}\beta(B(e,t))\sum_{j\in J' }\sup_{t'\in (0, c t^d] } t'^{-Q_*/d}\ee^{-\omega t'-b'(\abs{y_{t,j}}_*^d/t')^{1/(d-1)}} 
	\]
	and
	\[
	\sup_{t\in (0,c']}\beta(B(e,t))\sum_{j\in J' } \sup_{t'\in  [1/c,+\infty)} t'^{-Q_*/d}\ee^{-\omega t'-b'(\abs{y_{t,j}}_*^{d}/t')^{1/(d-1)}} 
	\]
	are finite. 
	Observe first that  the function $t'\mapsto t'^{-Q_*/d}\ee^{ -b'(\abs{y_{t,j}}_*^{d}/t')^{1/(d-1)}}$ attains its  maximum at the point $\alpha \abs{y_{t,j}}_*^d$, where $\alpha= \big(\frac{b' d}{Q_*(d-1)}\big)^{d-1}$. Consequently, 
	\[
	\sup_{t'\in (0, ct^d] } t'^{-Q_*/d}\ee^{-\omega t'-b'(\abs{y_{t,j}}_*^d/t')^{1/(d-1)}} \meg\begin{cases} \alpha^{-Q_*/d} \ee^{- b' \alpha^{-1/(d-1)}}\abs{y_{t,j}}_*^{-Q_*}  
		 & \text{if   $\alpha\abs{y_{t,j}}_*^d\meg c t^d$}\\
		 c^{-Q_*/d} t^{-Q_*}\ee^{ -b'c^{-1/(d-1)}(\abs{y_{t,j}}_*/t)^{d/(d-1)}}& \text{if   $\alpha \abs{y_{t,j}}_*^d> c t^d$}.
	\end{cases}
	\]
	Taking into account the fact that $\abs{y_{t,j}}_*\Meg 3 t$, we then see that there is a constant $C_4>0$ such that 
	\[
	\begin{split}
		\sup_{t'\in (0, ct^{d}] } t'^{-Q_*/d}\ee^{-\omega t'-b'(\abs{y_{t,j}}_*^{d}/t')^{1/(d-1)}}& \meg C_4  t^{-Q_*}\ee^{-b'' (\abs{y_{t,j}}_*/t)^{d/(d-1)}},
	\end{split}
	\]
	where $b''\coloneqq b'c^{-1/(d-1)}$.
	In a similar way, since the function $t'\mapsto -\omega t'-  b'(\abs{y_{t,j}}_*^{d}/t')^{1/(d-1)}$ attains its maximum at $ \alpha'\abs{y_{t,j}}_*$, where $\alpha'=\big( \frac{b'}{\omega(d-1)} \big)^{1-1/d}$, we see that
	\[
	\sup_{t'\in  [1/c,+\infty)} t'^{-Q_*/d}\ee^{-\omega t'-b'(\abs{y_{t,j}}_*^{d}/t')^{1/(d-1)}} \meg \begin{cases}
		 c^{Q_*/d}  \ee^{ -\omega^{1/d} d [b'/(d-1)]^{1-1/d}  \abs{y_{t,j}}_*  }& \text{if $\alpha'\abs{y_{t,j}}_* \Meg 1/c$}\\
		 c^{-Q_*/d}\ee^{-\omega /c-b'c^{1/(d-1)} \abs{y_{t,j}}_*^{d/(d-1)}}  & \text{if $\alpha'\abs{y_{t,j}}_* < 1/c$},
	\end{cases}
	\]
	so that there are $b''',C_5>0$ such that  
	\[
	\sup_{t'\in  [1/c,+\infty)} t'^{-Q_*/d}\ee^{-\omega t'-b'(\abs{y_{t,j}}_*^{d}/t')^{1/(d-1)}} \meg C_5\ee^{-b'''\omega^{1/d}\abs{y_{t,j}}_* } .
	\]
	In addition, there is a constant $C_6>0$ such that, for every $t\in (0,c']$,
	\[
	\begin{split}
	\beta(B(e,t))\sum_{j\in J_t} t^{-Q_*}  \ee^{-b'' (\abs{y_{t,j}}_*/t)^{d/(d-1)}}&\meg C_6 \sum_{j\in J_t}\int_{B(y_{t,j},t/2) } t^{-Q_*} \ee^{-2^{d/(d-1)}b''(\abs{y }_*/t)^{d/(d-1)}}\,\dd \beta(y)\\
		&\meg C_6\int_G t^{-Q_*} \ee^{-2^{-d/(d-1)}b''(\abs{y }_*/t)^{d/(d-1)}}\,\dd \beta(y),
	\end{split}
	\] 
	which is uniformly bounded for $t\in (0,c']$ by  Lemma~\ref{lem:3}.	Analogously, for every $t\in (0,c']$,
	\[
	\begin{split}
		\beta(B(e,t))\sum_{j\in J_t} \ee^{- b'''\omega^{1/d}  \abs{y_{t,j}}_* } &\meg C_6 \ee^{\omega^{1/d}b'''t/2} \sum_{j\in J_t}\int_{B(y_{t,j},t/2) } \ee^{-  b'''\omega^{1/d} \abs{y}_* } \,\dd \beta(y)\\
			&\meg C_6 \ee^{\omega^{1/d}b'''c'/2} \int_{G} \ee^{- b'''\omega^{1/d}  \abs{y}_* } \,\dd \beta(y),
	\end{split}
	\]
	which is finite if $\omega$ is sufficiently large. The assertion follows. 
\end{proof}
 
We now prove a series of technical lemmas which allow, under some assumptions, to replace the $\Cc^q(\mi)$ norm with an $L^{q,\infty}(\mi,\beta)$ norm (on a smaller set). Even though, in this context, these results will not have the far reaching consequences obtained in~\cite{FrazierJawerth} (as usual, because we lack a proper `discretization'), they will still be of some use.

\begin{deff}
	Take $q\in (0,\infty]$, $a>0$, $\eta\in (0,1)$,  and a positive measure $\mi$ on $(0,+\infty)$. Define
	\[
	(\Gc^{q}_{t} f)(x) \coloneqq \norm{  \chi_{(0,t]} f(\,\cdot\,,x) }_{L^q(  \mi)}
	\]
	and
	\[
	(\Gc^q f)(x)\coloneqq \sup_{t>0}(\Gc^{q}_{t} f)(x)=\norm{ f(\,\cdot\,,x) }_{L^q(  \mi)}
	\]
	for every $(\mi \otimes \beta)$-measurable function $f\colon (0,+\infty) \times G\to [0,+\infty]$ and for every $x\in G$.
	In addition, set
	\[
	(m^q_{\eta,a,t} f)(x)\coloneqq \sup\Set{\eps>0\colon  \beta_L(\Set{y\in B(x,t^a) \colon (\Gc^q_{t} f)(y)>\eps  })> \eta\beta(B(e,t^a))}
	\]
	and
	\[
	(m^q_{\eta,a} f)(x)\coloneqq \sup_{t\in (0,c]} (m^q_{\eta,a,t} f )(x)
	\]
	for every $f$ and $x$ as above, where $c\coloneqq \sup \supp \mi$.
\end{deff}
 
\begin{oss}\label{oss:5}
	Take $q\in (0,\infty]$, $a>0$, $\eta\in (0,1)$,  a positive measure $\mi$ on $(0,+\infty)$, and a $(\mi \otimes \beta)$-measurable function $f\colon (0,+\infty) \times G\to [0,+\infty]$. Then, $\Gc^q_t f, \Gc^q f, m^q_{\eta,a,t} f$, and $m^q_{\eta,a} f$ are $\beta$-measurable functions. 
\end{oss}

\begin{proof} 
	The measurability of the first two functions follows from Fubini's theorem. Then, observe that 
	\[
	\Set{x\in G\colon (m^q_{\eta,a,t} f)(x)> \alpha}= \bigcup_{\eps\in (\alpha,+\infty)} \Set{ x\in G\colon \beta_L(\Set{y\in B(x,t^a) \colon (\Gc^q_{t} f)(y)>\eps  })>  \eta\beta(B(e,t^a))  }
	\]
	for every $\alpha>0$.
	In addition, the function 
	\[
	G\ni x \mapsto \beta_L(\Set{y\in B(x,t^a) \colon (\Gc^q_{t} f)(y)>\eps  })=\int_{G} \chi_{(\Gc^q_{t} f)^{-1}((\eps,+\infty))}\chi_{B(x,t^a)}\,\dd \beta_L\in [0,+\infty)
	\]
	is continuous, since the function $G\ni x \mapsto \chi_{B(x,t^a)}=\chi_{B(e,t^a)}(x^{-1}\,\cdot\,)\in L^1(G)$ is continuous. Consequently,   $m^q_{\eta,a,t} f$ is lower semi-continuous, so that also $m^q_{\eta,a} f$ is lower semi-continuous.
\end{proof}

\begin{lem}\label{lem:59}
	Take $p\in (0,\infty)$, $q\in (0,\infty]$, $a>0$, $\eta\in (0,1)$,    and a positive measure $\mi$  on $(0,+\infty)$ with bounded support. Then there is a constant $C>0$ such that for every $(  \mi \otimes \beta)$-measurable function $f\colon (0,+\infty) \times G\to [0,+\infty]$ there is a $(\mi\otimes \beta)$-measurable subset $E$ of $(0,+\infty)\times G$ such that 
	\[
	\beta(\Set{y\in B(x,t^a)\colon (t,y)\in E})\Meg (1-\eta) \beta(B(e,t^a))
	\]
	for every $x\in G$ and for  every $t\in (0,\max \supp \mi]$, and such that
	\[
	 \norm{\chi_E f}_{L^{q,p}(\mi,\beta)}\meg \norm{m^q_{\eta,a} f}_{L^p(\beta)}\meg C\norm{f}_{L^{q,p}(\mi,\beta)}
	\]
	and
	\[
	\norm{\chi_E f}_{L^{q,\infty}(\mi,\beta)}\meg \norm{m^q_{\eta,a} f}_{L^\infty(\beta)}\meg C\norm{f}_{\Cc^{q}(\mi,a)}.
	\]
\end{lem}

Cf.~\cite[Proposition 5.5]{FrazierJawerth} for the classical case.

\begin{proof}
	We may assume that  $\beta=\beta_L$.
	Observe that, for every $\eps>0$,
	\[
	\Set{x\in G\colon (m^q_{\eta,a} f)(x)>\eps}\subseteq \Set{  x\in G\colon  \Mc^{(c^a)}(\chi_{ \Set{y\in G\colon (\Gc^q f)(y)>\eps} }  )>\eta },
	\]
	where $c=\max \supp \mi$ and $(\Mc^{(c^a)}g)(y)\coloneqq \sup_{r\in (0,c^a]}\dashint_{B(y,r)} \abs{g}\,\dd \beta$ for every $\beta$-measurable function $G$ and for every $y\in G$. Since $\Mc^{(c^a)}$ is of weak type $(1,1)$ (cf., e.g.,~\cite[Lemma 4]{Singular}), there is a constant $C_1>0$ such that
	\[
	\beta\big(\Set{x\in G\colon (m^q_{\eta,a} f)(x)>\eps}\big)\meg  C_1\beta(  \Set{y\in G\colon (\Gc^q f)(y)>\eps}).
	\]
	Consequently,
	\[
	\begin{split}
	\norm{m^q_{\eta,a} f}_{L^p(\beta)}^p&=p \int_0^{+\infty} t^{p-1} \beta\big(\Set{x\in G\colon (m^q_{\eta,a} f)(x)>\eps}\big)\,\dd t\\
		&\meg C_1 p \int_0^{+\infty} t^{p-1} \beta(  \Set{y\in G\colon (\Gc^q f)(y)>\eps})\,\dd t\\
		&=C_1 \norm{ \Gc^q f }_{L^p(\beta)}^p\\
		&=C_1 \norm{f}_{L^{q,p}(\mi,\beta)}^p.
	\end{split}
	\]
	Now, take $C_2>0$ so that $C_2^{-1}r^{Q_*}\meg \beta(B(e,r))\meg C_2 r^{Q_*}$ for every $r\in (0,c^a]$. If $q<\infty$, then
	\[
	\begin{split}
		\beta\Big(\Set{y\in B(x,t^a)\colon (\Gc^q_{t} f)(y)> (C_2/\eta)^{1/q} \norm{f}_{\Cc^q(\mi,a)}  }\Big)&\meg  \frac{\eta}{C_2  \norm{f}_{\Cc^q(\mi,a)}^q} \int_{B(x, t^a)} (\Gc^q_t f)^q\,\dd \beta\\
		&\meg \frac{\eta}{C_2 }  t^{aQ_*}  \\
		&\meg \eta\beta(B(e,t^a)),
	\end{split}
	\]
	so that $(m^q_{\eta,a,t} f)(x)\meg (C_2/\eta)^{1/q} \norm{f}_{\Cc^q(\mi,a)}$ for every $x\in G$ and for every $t\in (0,c]$. Consequently, $\norm{m^q_{\eta,a}f}_{L^\infty(\beta)}\meg (C_2/\eta)^{1/q}  \norm{f}_{\Cc^q(\mi,a)}$. 
	In a similar way, one may see that $\norm{m^\infty_{\eta,a} f}_{L^\infty(\beta)}\meg \norm{f}_{\Cc^\infty(\mi,a)}$.
	
	For the remaining inequalities, define, for every $j\in\N$,
	\[
	\theta_j(x)\coloneqq \sup\Set{ t\in (0,c]\colon ( \Gc_t^q f)(x)\meg (m^q_{\eta,a} f)(x) +2^{-j}}
	\]
	for every $x\in G$,	and observe that $\theta_j$ is a $\beta$-measurable function since 
	\[
	\Set{x\in G\colon \theta_j(x) \meg \alpha}= \bigcap_{k\in\N} \Set{x\in G\colon  ( \Gc_{\alpha+ 2^{-k}}^q f)(x)> (m^q_{\eta,a} f)(x)+2^{-j}}
	\]
	for every $\alpha\in (0,c)$. In addition, define
	\[
	E_j\coloneqq \Set{(t,x)\in  (0,\infty)\times G\colon  t< \theta_j(x)     },
	\]
	so that $E_j$ is $(\mi \otimes \beta)$-measurable. Observe that, for every $x\in G$ and for every $t\in (0,c]$,  
	\[
	[\Gc^q_t(\chi_{E_j} f) ](x)=\lim_{s\to (\theta_j(x))^-} ( \Gc_s^q f)(x)\meg (m^q_{\eta,a} f)(x)+2^{-j} 
	\] 
	if $t\Meg \theta_j(x)$, whereas
	\[
	[\Gc^q_t(\chi_{E_j} f) ](x)=  ( \Gc_t^q f)(x)\meg (m^q_{\eta,a} f)(x)+2^{-j} 
	\]
	otherwise.
	In addition,  
	\[
	\begin{split}
		\Set{y\in B(x,t^a)\colon \theta_j(y)\meg  t  }&\subseteq  \bigcap_{t'>t}\Set{y\in B(x,t^a)\colon ( \Gc_{t'}^q f)(y)> (m^q_{\eta,a} f)(y)+2^{-j}}\\
		&\subseteq \Set{y\in B(x,t^a)\colon ( \Gc_{t}^q f)(y)\Meg  (m^q_{\eta,a} f)(y)+2^{-j}} \\
		&\subseteq \Set{y\in B(x,t^a)\colon ( \Gc_{t}^q f)(y)\Meg  (m^q_{\eta,a,t} f)(y)+2^{-j}} , 
	\end{split}
	\]
	so that
	\[
	\beta(\Set{y\in B(x,t^a)\colon \theta_j(y)\meg  t  } )\meg \eta\beta(B(e,t^a)).
	\]
	Now, observe that $(\theta_j(x))$ is a decreasing sequence for every $x\in G$, so that $(E_j)$ is a decreasing sequence of $(\mi\otimes \beta)$-measurable subsets of $(0,+\infty)\times G$. Set $E\coloneqq \bigcap_j E_j$. Then,
	\[
	\beta(\Set{y\in B(x,t^a)\colon (t,y)\in E  } )=\lim_{j\to \infty}\beta(\Set{y\in B(x,t^a)\colon \theta_j(y)> t  } )\Meg (1-\eta) \beta(B(e,t^a)). 
	\]
	In addition, by Fatou's lemma,
	\[
	[\Gc^q_t(\chi_{E} f) ](x)\meg \lim_{j\to \infty} [\Gc^q_t(\chi_{E_j} f) ](x)\meg (m^q_{\eta,a} f)(x) 
	\]
	for every $t\in (0,c]$ and for every $x\in G$,
	so that
	\[
	\norm{\chi_E f}_{L^{q,p}(\mi,\beta)}\meg \norm{m^q_{\eta,a} f}_{L^p(\beta)}
	\]
	and
	\[
	\norm{\chi_E f}_{L^{q,\infty}(\mi,\beta)}\meg \norm{m^q_{\eta,a} f}_{L^\infty(\beta)}.
	\]
	The assertion follows.
\end{proof}

\begin{lem}\label{lem:60}
	Take $p\in (0,\infty)$, $q\in (0,\infty]$, $\eta\in (0,1)$,   a positive Radon measure  $\mi $ on $(0,+\infty)$ with bounded support, and $a,c,c'>0$. Then, there is a constant $C>0$ such that 
	\[
	\norm{ f}_{L^{q,p}  (\mi,\beta)}\meg C \norm{\chi_E g}_{L^{q,p}  (\mi,\beta)}
	\]
	and
	\[
	\norm{f}_{\Cc^{q} (\mi,a)}\meg C \norm{\chi_E g}_{L^{q,\infty} (\mi,\beta)}
	\]
	for every $(\mi\otimes \beta)$-measurable functions $f,g\colon (0,+\infty)\times G\to [0,+\infty]$  and for every $(\mi\otimes \beta)$-measurable subset $E$ of $(0,\infty)\times G$ such that 
	\begin{equation}\label{eq:2}
	f(t,x)\meg c g(t,y)
	\end{equation}
	for $(\mi\otimes \beta \otimes \beta)$-almost every $(t,x,y)\in (0,+\infty)\times G\times G$ such that $(t,y)\in E$ and $d(x,y)<\kappa t^a$, and such that
	\[
	\beta  (\Set{y\in B(x,c' t^{a})\colon (t,y)\in E }  )\Meg \eta \beta(B(x,c' t^{a}))  
	\]
	for $(\mi\otimes \beta)$-almost every $(t,x)\in (0,+\infty)\times G$.
\end{lem}

Cf.~\cite[Proposition 2.7]{FrazierJawerth} for the classical case.

	In particular, given $b>0$, $d>1$, and a $\mi$-measurable function $\kappa \colon (0,+\infty)\to (0,+\infty)$ such that $\sup_{t>0} t/\kappa(t)$ is finite,   the statement applies to 
	\[
	f(t,x)= [T_{b,\kappa(t),d}h(t,\,\cdot\,)](x)\qquad \text{and}\qquad  g(t,x)=[T_{2^{1/(d-1)}b,\kappa(t),d}h(t,\,\cdot\,)](x)
	\]
	for every $(\mi\otimes \beta)$-measurable function $h\colon (0,+\infty)\times G\to [0,+\infty]$, 
	provided that $a=1/d$ and $c$ is sufficiently large, thanks to Lemma~\ref{lem:32}.
	
	Notice that, if $f$ and $g$ satisfy~\eqref{eq:2} for $(\mi\otimes \beta \otimes \beta)$-almost every $(t,x,y)\in (0,+\infty)\times G\times G$ such that  $d(x,y)<\kappa t^a$, then for every $b>0$ there is a constant $c'>0$ such that $f(t,\,\cdot\,)\meg c' T_{b,t,1/a}(g(t,\,\cdot\,))$ $\beta$-almost everywhere for $\mi$-almost every $t>0$. Conversely, if $f(t,\,\cdot\,)\meg c'' T_{b,t,1/a}(g'(t,\,\cdot\,))$ $\beta$-almost everywhere for $\mi$-almost every $t>0$, then one may take $g(t,x)= T_{b/2,t,1/a}(g'(t,\,\cdot\,))(x)$ in the above statement, thanks to Lemma~\ref{lem:32}. In other words, this formulation essentially `contains' those of Lemma~\ref{lem:26bis},~\ref{lem:25b}, and~\ref{lem:25bisb} below. We chose different formulations for convenience in the applications. 

\begin{proof} 
	Set $p_0\coloneqq \min(p/2,q/2,1)$, and observe that  
	\[
	f(t,x)\meg c \eta^{-1/p_0}\left( \dashint_{ B(x,c' t^{a})} \chi_E(t,y)g(t,y)^{p_0}\,\dd \beta(y)\right) ^{1/p_0} 
	\]
	for $(\mi\otimes \beta)$-almost every $(t,x)\in (0,+\infty)\times G$. By the (vector-valued) boundedness of the truncated Hardy--Littlewood maximal function (cf.~\cite[Theorem 13]{Singular}), we then see that there is a constant $C_1>0$ such that
	\[
	\norm{f}_{L^{q,p}(\mi,\beta)}\meg C_1 \norm{\chi_E g}_{L^{q,p} (\mi,\beta)}.
	\]
	In a similar way, we see that there is a constant $C_2>0$ such that
	\[
	\begin{split}
		\norm{\chi_{(0,t]\times B(x,t^{a})}  f }_{L^q (\mi\otimes \beta)}& \meg C_2\norm{\chi_{(0,t]\times B(x,(1+c')t^{a})} \chi_E g }_{L^q (\mi\otimes \beta)}\\
		&  \meg C_2 \beta(e, (1+c') t^{a})^{1/q}  \norm{  \chi_E g }_{L^{q,\infty} (\mi, \beta)}
	\end{split}
	\]
	for  every $(t,x)\in (0,c']\times G$, where $c'=\max \supp \mi$. It then follows that there is a constant $C_3>0$ such that
	\[
	t^{-aQ_*/q}\norm{\chi_{(0,t]\times B(x,t^{a})} f }_{L^q (\mi\otimes \beta)}\meg C_3 \norm{\chi_E g}_{L^{q,\infty} (\mi,\beta)},
	\]
	for every $(t,x)\in (0,c']\times G$.
	The assertion follows. 
\end{proof}

\section{Some Technical Lemmas}\label{sec:4}

We now pass to the modifications that are needed to study the spaces $F^{\infty,q}_\alpha(\beta)$. We begin with the technical lemmas.
The next two lemmas replace~\cite[Lemmas 6.7 and 6.8]{BCP}.

\begin{lem}\label{lem:26bis}
	Take $\alpha\in \R$, $ q\in [1,\infty]$,   $\mi\in \Mc_\Car$,   $\nu\in \Mc_\RC$, $ c,d>1$, and $b>0$. Then, there are $\kappa>1$ and $C>0$ such that, for every $(\mi\otimes \beta)$-measurable function $f\colon (0,+\infty)\times G\to [0,+\infty]$ and for every $(\nu\otimes \beta)$-measurable function $g\colon (0,+\infty)\times G\to [0,+\infty]$  such that    $ f(t,\,\cdot\,)\meg c  T_{b,c t,d} [g(s,\,\cdot\,)]$ $\beta$-almost everywhere for $(\mi\otimes \nu)$-almost every $(t,s)\in (0,+\infty)^2$ such that $t/\kappa^2\meg s\meg t/\kappa $,
	\[
	\norm*{ t^{\alpha/d} f(t,x) }_{\Cc^q_{(t,x)}(\mi,1/d)}\meg C \norm*{ t^{\alpha/d} g(t,x) }_{\Cc^q_{(t,x)}(\nu,1/d)}.
	\] 
\end{lem}
 
Notice that the assumption on $f$ and $g$ is trivially met if $ f(t,\,\cdot\,) \meg c' T_{b,c'(t-s),d}[g(s,\,\cdot\,)]$ for a suitable $c'>0$ and for $(\mi\otimes \nu)$-almost every $(t,s)$ such that $t>s>0$.  

\begin{proof}
	We may assume that $\beta=\beta_L$.  Set $c'\coloneqq \max \supp \mi$ and $c''\coloneqq \max \supp \nu$.  
	Take $\eps\in (0,1)$ and $N\in\N$ so that $N\Meg 1$, $\supp \mi\subseteq (0,\eps^{-N}]$,
	\[
	\nu((\eps t, t])\Meg \eps^N
	\]
	for every $t\in (0,\eps^N]$, and
	\[
	\mi((\eps t,t])\meg \eps^{-N}
	\]
	for every $t>0$ (cf.~\cite[Lemma 6.2]{BCP}), and set $\kappa\coloneqq \eps^{-2N-1}$.	In addition, write $f_t$ and $g_t$ instead of $f(t,\,\cdot\,)$ and $g(t,\,\cdot\,)$, respectively, for every $t>0$.
	Observe that, for $(\mi\otimes \nu)$-almost every $(t,s)\in [\eps^{-N+h+1},\eps^{-N+h}]\times[\eps^{N+h+1},\eps^{N+h}]$, 
	\[
	 {f_t}\meg c T_{b,c t,d}  g_s\meg c \eps^{-Q_*/d}  T_{b,c\eps^{-N+h},d} g_s
	\]
	since $t/\kappa^2\meg s\meg t/\kappa $. In particular,
	\[
	 {f_t} \meg c \eps^{-Q_*/d} \underset{s\in [\eps^{N+h+1},\eps^{N+h}]}{\essmin_\nu} T_{b,c\eps^{-N+h},d} g_s
	\]
	for $\mi$-almost every $t\in [\eps^{-N+h+1},\eps^{-N+h}]$. For every $h\in\N$, set $\mi_h\coloneqq \chi_{(\eps^{-N+h+1},\eps^{-N+h}]}\cdot \mi$ and $\nu_h\coloneqq \chi_{(\eps^{N+h+1},\eps^{N+h}]}\cdot \nu$
	Then, Lemmas~\ref{lem:63} and~\ref{lem:2bis} shows that there is a constant $C_1>0$ such that, for every $t\in (0,c']$ and for every $x\in G$, choosing $h_0\in \N$ so that $t\in (\eps^{-N+h_0+1},\eps^{-N+h_0}]$,
	\[
	\begin{split}
		&t^{-Q_*/(dq)}\norm*{\chi_{(0,t ]\times B(x,t^{1/d})}(s,y) s^{\alpha/d} f_s(y) }_{L^q_{(s,y)}(\mi \otimes \beta)}\\
		&\qquad = t^{-Q_*/(dq)}\norm*{\norm*{  \chi_{(0,t ]\times B(x,t^{1/d})}(s,y) s^{\alpha/d} f_s(y) }_{L^q_{(s,y)}( \mi _h\otimes \beta)}}_{\ell_h^q(\N)}\\
		&\qquad \meg  t^{-Q_*/(dq)}\norm*{\norm*{  \chi_{B(x,t^{1/d})}(y) s^{\alpha/d} f_s(y) }_{L^q_{(s,y)}( \mi _h\otimes \beta)}}_{\ell_h^q(\N+h_0)}\\
		&\qquad\meg c  t^{-Q_*/(dq)}  \eps^{-(N\alpha+\alpha_-+Q_*)/d -N/q}\norm*{\norm{ \eps^{h\alpha/d}\chi_{  B(x,t^{1/d})}(y) \underset{s\in [\eps^{N+h+1},\eps^{N+h}]}{\essmin_\nu} T_{b,c\eps^{-N+h},d}  g_{s}(y)  }_{\ell^q_h(\N+h_0)}}_{L^q_y(\beta)}\\
		&\qquad\meg c   t^{-Q_*/(dq)} \eps^{-(N\alpha+\alpha_-+Q_*)/d -2N/q}\norm*{\eps^{h\alpha/d}\norm{ \chi_{ B(x,t^{1/d})}(y) T_{b,c\eps^{-N+h},d}  g_{s}(y) }_{L^q_{(s,y)}(\nu_h\otimes \beta)} }_{\ell^q_h(\N+h_0)} \\
		&\qquad\meg c    t^{-Q_*/(dq)}  \eps^{-(2N\alpha+\abs{\alpha}+2Q_*)/d -2N/q }\norm*{ \chi_{(0,t ]\times B(x,t^{1/d})}(s,y) s^{\alpha/d} T_{b,c s\eps^{-2N-1},d}  g_{s}(y) }_{L^q_{(s,y)}(\nu\otimes \beta  )}   \\
		&\qquad\meg C_2  \norm{ s^{\alpha/d}    g_{s }(y)  }_{\Cc^q_{(s,y)}(\nu,1/d)},
	\end{split}
	\]
	since $s \eps^{-2N}\meg \eps^{-N+h}\meg s \eps^{-2N-1}$ for every $s\in [\eps^{N+h+1},\eps^{N+h}]$. The assertion follows thanks to the arbitrariness of $x\in G$ and $t\in (0,c']$.
\end{proof}

\begin{lem}\label{lem:27bis}
	Take $\alpha\in \R$,    $ q\in [1,\infty]$, $m\in \N$,  $\mu\in \Mc_\RC$,  and $(f_t)_{t>0}\in \Sc'(G)^{(0,+\infty)}$ such that $f_{s+t}=\ee^{-s \Lc} f_t$ for every $t,s>0$, and such that   $\norm{t^{\alpha/\grado}  (\Lc^m f_{t})(x) }_{\Cc^{q}_{(t,x)}(\mi,1/\grado)} <\infty$. Then, there is a unique $f\in \Sc'(G)$ such that $f_t=\ee^{-t\Lc} f$ for every $t>0$.
\end{lem}

\begin{proof}
	By Lemma~\ref{lem:26bis} (applied with $f(t,x)=f_t(x)$ and $g(t,x)=f_{t/2}(x)$), we may reduce to the case in which $\mi=\sum_{j\in\N} 2^{-j}\delta_j$. 
	take $x\in G$ and $j\in \N$, and observe that by Theorem~\ref{teo:7} and Lemma~\ref{lem:32} we may find $C_1>1$ and $b>0$ such that
	\[
	\abs{(\Lc^m f_{2^{-j}})(x)}\meg C_1 (T_{b, 2^{-j-1}} \Lc^m f_{2^{-j-1}})(x)\meg C_1^2 (T_{b/2, 2^{-j-1}} \Lc^m f_{2^{-j-1}} f)(y)
	\]
	for every $y\in B(x,2^{-j/\grado })$. Consequently, there is a constant $C_2>0$ such that
	\[
	\begin{split}
		2^{-(j\alpha+Q_*)/\grado}\abs{(\Lc^m f_{2^{-j}})(x)} &\meg C_2 2^{-j\alpha/\grado}\norm{\chi_{B(x,2^{-j/\grado})} T_{b/2, 2^{-j-1}} \Lc^m f_{2^{-j-1}} }_{L^q(\beta)}\\
		&\meg C_2 \norm{ s^{\alpha/\grado} \chi_{ (0,2^{-j}]\times B(x,2^{-j/\grado})}(s,y) (T_{b/2, s/2} \Lc^m f_{s/2} )(y) }_{L^q_{(s,y)}(\mi \otimes \beta)}\\
		&\meg C_2 \norm{ s^{\alpha/\grado} (T_{b/2, s/2} \Lc^m f_{s/2} )(y)}_{\Cc^q_{(s,y)}(\mi)}.
	\end{split}
	\]
	By Lemmas~\ref{lem:61} and~\ref{lem:2bis} it then follows that $
	\norm{t^{\alpha/\grado} (\Lc^m f_{t})(x)}_{L^\infty_{(t,x)}(\mi \otimes \beta)}<\infty$,	so that the assertion follows from~\cite[Lemma 6.8]{BCP}.
\end{proof}

The next two results provide suitable analogues of Lemma~\ref{lem:25} and~\ref{lem:25bis}, which cannot be applied directly in this context. They are inspired by~\cite[Lemma 4]{Rychkov}. 

\begin{lem}\label{lem:25b}
	Take $q\in [1,\infty]$, $c, d>1$, $ b,\eta,\delta>0$,   $\mi\in \Mc_\Samp$, $\nu\in \Mc_\Car$. Then, there are $\kappa>1$ and $C>0$ such that
	\[
	\norm*{  \int_0^{+\infty} \frac{s^\delta t^\eta}{(s+t)^{\delta+\eta}}  f(s,x) \,\dd \mi(s)}_{\Cc^q_{(t,x)}(\nu,1/d)}\meg  C \norm{  g }_{\Cc^q (\mi,1/d)}
	\] 
	for every $(\mi\otimes \beta)$-measurable functions $f,g\colon (0,+\infty)\times G\to [0,+\infty]$  such that   $ f(t,\,\cdot\,) \meg c  T_{b,c t,d} [g(s,\,\cdot\,)]$ $\beta$-almost everywhere for $(\mi\otimes \mi)$-almost every $(t,s)\in (0,+\infty)^2$ with $t/\kappa\meg s\meg t$, and such that $ f(t,\,\cdot\,) \meg c  T_{b,c t,d} [g(t,\,\cdot\,)]$ for $\mi$-almost every $t>0$. 
\end{lem}

This result applies, for example, to $f(t,x)=\abs{(W^{(m)}_t h)(x)}$ and $g(t,x)= \abs{(W^{(m)}_{\eps t} h)(x)}$ ($h\in \Sc'(G)$, $m\in\N$, $\eps\in (0,1)$), if $b$ is sufficiently small, but also to $f(t,x)=(T_{b',t} W^{(m)}_t h)(x)$ and $ g(t,x)= (T_{b'',\kappa' (1-\eps) t} W^{(m)}_{\eps t} h)(x)$ ($h\in \Sc'(G)$, $m\in\N$, $\eps\in (0,1)$, $\kappa'>1$), for suitably related $b,b',b''$. 

\begin{proof}
	We may assume that $\beta=\beta_L$. Set $c'\coloneqq \max\supp (\mi+\nu)$.
	Observe  that, by Lemmas~\ref{lem:25} and~\ref{lem:63}, there is a constant $C_1>1$ such that, for every $x'\in G$ and  for every $t'\in (0,c']$,
	\[
	\begin{split}
		&\norm*{ \chi_{(0,t']\times B(x',t'^{1/d})}(t,x)  \int_0^{+\infty} \frac{s^\delta t^\eta}{(s+t)^{\delta+\eta}} f(s,x) \,\dd \mi(s)}_{L^q_{(t,x)}(\beta \otimes\nu)}\\
		&\qquad \meg \norm*{ \chi_{(0,t']\times B(x',t'^{1/d})}(t,x)  \int_{t'}^{+\infty} \frac{s^\delta t^\eta}{(s+t)^{\delta+\eta}}  {f (s,x)}\,\dd \mi(s)}_{L^q_{(t,x)}(\beta \otimes\nu)}\\
		&\qquad \qquad+ C_1 \norm*{ \chi_{(0,t']\times B(x',t'^{1/d})}(t,x)   f(t,x) }_{L^q_{(t,x)}(\beta \otimes\mi)}\\
		&\qquad \meg \norm*{ \chi_{ (0,t']}( t)\int_{t'}^{+\infty} (t/s)^{\eta }   \norm{\chi_{B(x',t'^{1/d})} f(s,\,\cdot\,)  }_{L^q(\beta)}\,\dd \mi(s)}_{L^q_{t}( \nu)}\\
		&\qquad \qquad+ C_1^2 t'^{Q_*/(dq)} \norm*{   f }_{\Cc^q (\mi,1/d)} .
	\end{split}
	\]
	Since $ f(t,\,\cdot\,) \meg c  T_{b,c t,d} [g(t,\,\cdot\,)] $ for $\mi$-almost every $t>0$,  by means of Lemma~\ref{lem:2bis} we see that there is a constant $C_2>0$ such that $\norm*{   f }_{\Cc^q (\mi,1/d)}\meg C_2\norm*{g}_{\Cc^q (\mi,1/d)}$. 
	In addition,  by  Lemma~\ref{lem:32}, there is $ C_3>0$ such that  
	\[
	{f(s,x)}\meg c  (T_{b,c s,d} g(s',\,\cdot\,))(x)\meg  C_3  (T_{b/2,c s,d} g(s',\,\cdot\,))(y)
	\]
	for  $\mi$-almost every $s>0$, for $\mi$-almost every $s'\in [s/\kappa,s]$, for $\beta$-almost every $x\in G$,  and for every $y\in B(x,2 s^{1/d} )$. Consequently, there is a constant $C_4>0$ such that
	\[
	\begin{split}
		t'^{-Q_*/(dq)}\norm{\chi_{B(x',t'^{1/d})}f(s,\,\cdot\,) }_{L^q(\beta)}&\meg t'^{-Q_*/(dq)}\beta(B(e,t'^{1/d}))^{1/q}C_3  \min_{ B(x', 2 s^{1/d}-t'^{1/d} )}T_{b/2,c s,d} g(s',\,\cdot\,) \\
			&\meg C_4 s^{-Q_*/(dq)} \norm{\chi_{B(x',s'^{1/d})}  T_{b/2,\kappa s',d}g(s',\,\cdot\,) }_{L^q(\beta)}
	\end{split}
	\]
	for every $t'\in (0,c']$, for every $x'\in G$,  for $\mi$-almost every $s\Meg t'$, and for $\mi$-almost every $s'\in [s/\kappa,s]$.
	Now, arguing as in the proof of  Lemma~\ref{lem:26bis} and choosing $\kappa$ sufficiently large (so that $\mi([s/\kappa,s])$ is bounded from below for $s\in (0,\max\supp \mi]$), we deduce that there is a constant $C_5>0$ such that
	\[
	t'^{ -Q_*/(dq)} \norm{\chi_{B(x',t'^{1/d})} f(s,\,\cdot\,) }_{L^q(\beta)}\meg C_5 \norm{   (T_{b/2,\kappa s',d}g(s',\,\cdot\,) )(y) }_{\Cc^q_{(s',y)}(\mi,1/d)}
	\] 
	for every $t'\in (0,c']$, for every $x'\in G$, and for $\mi$-almost every  $s\Meg t'$.
	By Lemma~\ref{lem:2bis}  we then see that there is a constant $C_6>0$ such that
	\[
	\begin{split}
		t'^{-Q_*/(dq)}\norm{\chi_{B(x',t'^{1/d})} f(s,\,\cdot\,) }_{L^q(\beta)}&\meg C_6 \norm{  g  }_{\Cc^q (\mi,1/d)}\\ 
	\end{split}
	\]
	for every $c'\Meg s\Meg t'>0$ and for every $x'\in G$.
	The   assertion will then follow if we show that 
	\[
	\begin{split}
		\norm*{ \chi_{ (0,t']}( t)\int_{t'}^{+\infty} (t/s)^{\eta }  \,\dd \mi(s)}_{L^q_{t}( \nu)}&=\norm*{ \chi_{ (0,t']}( t) (t/t')^{\eta } \int_{t'}^{+\infty} (t'/s)^{\eta }  \,\dd \mi(s)}_{L^q_{t}( \nu)}\\
		&\meg 4^{\eta }\norm*{  \bigg(\frac{\sqrt{t' t}}{t'+t}\bigg)^{\eta } \int_{0}^{+\infty} \bigg(\frac{\sqrt{t' s}}{t'+s}\bigg)^{\eta }  \,\dd \mi(s)}_{L^q_{t}( \nu)}
	\end{split}
	\]
	is uniformly bounded for $t'\in(0,c']$. This follows from the proof of~\cite[Lemma 6.4]{BCP}.  
\end{proof}

\begin{lem}\label{lem:25bisb}
	Take $q\in[1,\infty]$, $c,d>1$, $a,b,\eta,\gamma>0$ and  let $\mi$ be a Haar measure on $(0,+\infty)$. Take $\nu\in \Mc_\Car$, and assume that either $\eta\Meg 1$ or $\nu\in L^\infty(\mi)\cdot \mi$.	 
	Then, there are $\kappa>1$ and   $C>0$ such that 
	\[
	\norm*{\int_0^\infty t^\eta s^\delta   f(t+s,x)\,\dd \mi(s)  }_{\Cc^q_{(t,x)}(\nu,1/d)} \meg C \norm{ t^{\delta+\eta}g(t,x)}_{\Cc^q_{(t,x)}(\mi,1/d)}
	\]
	for every $(\mi\otimes \beta)$-measurable functions $f,g\colon (0,+\infty)\times G\to [0,+\infty]$  such that   $ f(t,\,\cdot\,)\meg c  T_{b,c t,d} [g(s,\,\cdot\,)]$ $\beta$-almost everywhere for $(\mi\otimes \mi)$-almost every $(t,s)\in (0,+\infty)^2$ with $t/\kappa\meg s\meg t$,  such that $ f(t,\,\cdot\,)\meg c  T_{b,c t,d} [g(t,\,\cdot\,)]$ for $\beta$-almost every $t>0$, and such that $f(t,\,\cdot\,)=0$ for $\mi$-almost every $t>a$. 
\end{lem}

\begin{proof}
	We may assume that $\beta=\beta_L$. Set $c'\coloneqq \max(a,\max \supp \nu)$. 
	Observe  that, by Lemmas~\ref{lem:25} and~\ref{lem:63}, there is a constant $C_1>1$ such that, for every $x'\in G$ and  for every $t'\in (0,c']$,  
	\[
	\begin{split}
		&\norm*{ \chi_{(0,t'] \times B(x',t'^{1/d})}(t,x)  \int_0^{+\infty}  s^\delta t^\eta    f(s+t,x)\,\dd \mi(s)}_{L^q_{(t,x)}(\beta \otimes\nu)}\\
		&\qquad \meg \norm*{ \chi_{(0,t'] \times B(x',t'^{1/d})}(t,x)  \int_{t'}^{+\infty}  s^\delta t^\eta  f(s+t,x)\,\dd \mi(s)}_{L^q_{(t,x)}(\beta \otimes\nu)}\\
		&\qquad \qquad+ C_1 \norm*{ \chi_{(0,2t'] \times B(x',t'^{1/d})}(t,x)  t^{\delta+\eta}  f (t,x) }_{L^q_{(t,x)}(\mi \otimes \beta)}\\
		&\qquad \meg \norm*{ \chi_{ (0,t']}( t)\int_{t'}^{c'} s^\delta t^\eta \norm{\chi_{B(x',t'^{1/d})}  f(s+t,\,\cdot\,)  }_{L^q(\beta)}\,\dd \mi(s)}_{L^q_{t}( \nu)}\\
		&\qquad \qquad+ C_1^2 t'^{Q_*/(dq)} \norm*{  t^{\delta+\eta} f (t,x) }_{\Cc^q_{(t,x)}(\mi)} .
	\end{split}
	\]
	As in the proof of Lemma~\ref{lem:25b} we see that there is a constant $C_2>0$ such that $\norm*{  t^{\delta+\eta} f (t,x) }_{\Cc^q_{(t,x)}(\mi)}\meg C_2 \norm*{  t^{\delta+\eta} g(t,x) }_{\Cc^q_{(t,x)}(\mi)}$ and such that
	\[
	(s+t)^{\delta+\eta}\norm{\chi_{B(x',t'^{1/d})} f(s+t,\,\cdot\,)  }_{L^q(\beta)}\meg C_2 t'^{Q_*/(dq)} \norm{ s'^{\delta+\eta} f (s',x)}_{\Cc^q_{(s',x)}(\mi)}
	\]
	for every $t'\in (0,c']$, for every $x'\in G$,  for every $t\in (0,t']$, and for $\mi$-almost every   $s\in [t',c']$, provided that $\kappa$ is sufficiently large. It then suffices to prove that
	\[
	\sup_{t'\in (0,c']} \norm*{ \chi_{ (0,t']}( t)\int_{t'}^{c'} \frac{s^\delta t^\eta}{(s+t)^{\delta+\eta}} \,\dd \mi(s)}_{L^q_{t}( \nu)}
	\]
	is finite. This may be done arguing as in the proof of Lemma~\ref{lem:25b}.
\end{proof}

\begin{lem}\label{lem:8b}
	Take $\alpha\in \R$, $q\in [1,\infty]$, $\lambda \Meg 0$, $\eps \in (0,1]$,   $t_0>0$, $t_1\Meg 0$, $\mi\in \Mc_\Car$, and $\nu\in \Mc_\RC$.  Take $m  \in\N$  and $m'\in (\alpha/\dd_{\Lc'},+\infty) $.
	
	Let $\Lc'$ be a weighted subcoercive operator on $G$ with degree $\grado'$, and define $W'^{(m'),*}_{t,\eps}$ as in Definition~\ref{def:2} replacing $\Lc$ with $\Lc'$. Then, there is a constant $C>0$ such that, for every $X\in U_\lambda$, and for every $f\in \Sc'(G)$
		\[
			\begin{split}
				&\norm{\ee^{-t_0\Lc'}X f}_{L^\infty(\beta)}+ \norm{t^{-\alpha/\dd_{\Lc'} }(W'^{(m'),*}_{t,\eps} X f)(x)}_{\Cc^q_{(t,x)}(\mi,1/\grado')} \\
				&\qquad\meg C\abs{X} \norm{\ee^{-t_1\Lc}f}_{L^\infty(\beta)}+C\abs{X} \norm{t^{-\alpha/\grado }(W^{(m)}_{t} f)(x)}_{\Cc^q_{(t,x)}(\nu,1/\grado)}.
			\end{split}
		\]
\end{lem}

Notice that we may have as well allowed $t_0\Meg 0$ if $\alpha>0$, but this would have required another modification in the proof. We consequently preferred to get this improvement as a consequence of Proposition~\ref{prop:2b}.

\begin{proof}
The proof is essentially a repetition (with inessential modifications) of the proof of~\cite[Proposition 6.9]{BCP}, with the replacement of~\cite[Lemmas 5.11, 6.7, 6.8, and 6.4]{BCP}  with Lemmas~\ref{lem:2bis},~\ref{lem:26bis},~\ref{lem:27bis}, and~\ref{lem:25b}, respectively. The only relevant difference arises when trying to estimate $\norm{\ee^{- t_0 \Lc'} X f}_{L^\infty(\beta)}$, since in this case the trick used in the proof of~\cite[Proposition 6.9]{BCP}   leads to an estimate of the form
\[
\sup_{x\in G} \norm{\chi_{B(x,t_0^{1/\grado'})} \ee^{-t_0\Lc'} X f }_{L^q(\beta)}\meg C_1  \abs{X} \norm{\ee^{-t_1\Lc}f}_{L^\infty(\beta)}+C_1\abs{X} \norm{t^{-\alpha/\grado }(W^{(m)}_{t} f)(x)}_{\Cc^q_{(t,x)}(\nu,1/\grado)},
\]
which is too weak for our purposes. However, a careful inspection of the proof shows that we may in fact get to an estimate of the form
\[
\sup_{x\in G} \norm{\chi_{B(x,t_0^{1/\grado'})} T_{b,t_0/2,\grado'}\ee^{-(t_0/2)\Lc'} X f }_{L^q(\beta)}\meg C_2 \abs{X} \norm{\ee^{-t_1\Lc}f}_{L^\infty(\beta)}+C_2\abs{X} \norm{t^{-\alpha/\grado }(W^{(m)}_{t} f)(x)}_{\Cc^q_{(t,x)}(\nu,1/\grado)}
\]
for suitable $b,C_2>0$. By means of Lemma~\ref{lem:32} one may then observe that there is a constant $C_3>0$ such that
\[
\norm{T_{2b,t_0/2,\grado'}\ee^{-(t_0/2)\Lc'} X f }_{L^\infty(\beta)}\meg C_3\sup_{x\in G} \norm{\chi_{B(x,t_0^{1/\grado'})} T_{b,t_0/2,\grado'}\ee^{-(t_0/2)\Lc'} X f }_{L^q(\beta)}.
\]
The assertion then follows by Theorem~\ref{teo:7}, which ensures that $b$ may be chosen so small that $\abs{\ee^{(t_0/2)\Lc'}\delta_e}\meg C_4 p_{2b,t_0/2,\grado'}$ for a suitable constant $C_4>0$.
\end{proof}

\section{The Spaces $F^{\infty,q}_\alpha(\beta)$}\label{sec:5}

We may now finally state the definition of the spaces $F^{\infty,q}_\alpha(\beta)$. 

\begin{deff}
	Take $q\in [1,\infty]$, $\alpha\in \R$, $c>0$, $\mi \in \Mc_\Samp$, $m\in \N$ with $m>\alpha/\grado$, and $t_0>0$.   Define $F^{\infty,q}_\alpha(\beta)$ as the space of $f\in \Sc'(G)$ such that
	\[
	\norm{\ee^{-t_0\Lc}f}_{L^\infty(\beta)}+\norm{ t^{-\alpha/\grado } (W^{(m)}_{t} f)(x) }_{\Cc^q_{(t,x)}(\mi,1/\grado)}
	\]
	is finite, endowed with the corresponding norm. We define $\mathring F^{\infty,q}_\alpha(\beta)$ as the the closure of $C^\infty_c(G)$ in $F^{\infty,q}_\alpha(\beta)$. 
\end{deff}

Notice that the definition of the spaces $F^{\infty,q}_\alpha(\beta)$ and $\mathring F^{\infty,q}_\alpha(\beta)$ does not depend on  $\beta$, $\mi$, $m$, $t_0$, or $\Lc$, as a consequence of Lemma~\ref{lem:8b} (and the definition of $\Cc^q(\mi,a)$). We write  $F^{\infty,q}_\alpha(\beta)$ instead of $F^{\infty,q}_\alpha$ or $F^{\infty,q}_\alpha(G)$ only for notational consistency.  

We collect in the following result the analogues of~\cite[Proposition 7.6, Lemmas 7.5 and 7.7, and Theorems 7.8 and 7.9]{BCP}, which may be essentially proved in the same way using the above substitutions. We shall only highlight the relevant differences.

\begin{teo}\label{teo:15}
	Take $q\in [1,\infty]$ and $\alpha\in \R$. Then, the following hold:
	\begin{enumerate}
		\item[\textnormal{(1)}] $F^{\infty,q}_\alpha(\beta)$ and $\mathring F^{\infty,q}_\alpha(\beta)$ are Banach spaces;
		
		\item[\textnormal{(2)}] given $m\in \N$ with $m>\alpha/\grado$ and $\mi\in \Mc_\Samp$, if we define $\Fc_{m }$   as the space of $(f_t)\in \Sc'(G)^{[0,+\infty)}$ such that $t^{m }\Lc^m f_0=\ee^{-(1-t )\Lc} f_t$ for $t\in (0,1]$ while $f_t=t^{m } \Lc^m \ee^{(t-1)\Lc}f_0$ for $t\in [1,\infty)$, and set  $\theta_\alpha\colon (0,+\infty) \times G\ni (t,x)\mapsto t^{\alpha/\grado}\in (0,+\infty)$, then the mapping 
		\[
		T_{m }\colon \Sc'(G)\to \Fc_m
		\]
		such that $(T_{m })_0 f\coloneqq \ee^{-\Lc}f$ and $(T_{m } f)_t\coloneqq W^{(m)}_{t }f$ for $t>0$ and for every $f\in \Sc'(G)$,
		 induces an isomorphism of $F^{p,q}_\alpha(\beta)$ onto $[L^\infty(\beta)\oplus(\theta_\alpha \Cc^q(\mi,1/\grado))]\cap \Fc_{m}$;\footnote{Here, by abuse of notation, we write $L^\infty(\beta)\oplus(\theta_\alpha \Cc^q(\mi,1/\grado))$ to denote the space of $(f_t)_{t\Meg 0}$ such that $f_0\in L^\infty(\beta)$ and the mapping $(t,x)\mapsto f_t(x)$ belongs to $\theta_\alpha \Cc^q(\mi,1/\grado)$ or, equivalently, such that $f_0\in L^\infty(\beta)$ and the mapping $(t,x)\mapsto t^{-\alpha/\grado} f_t(x)$ belongs to $\Cc^q(\mi,1/\grado)$.}
		 
		 \item[\textnormal{(3)}] with the notation of \textnormal{(2)}, if we assume that $\mi=\sum_{j\in \N} \delta_{\eps^j}$ for some $\eps\in (0,1)$,
		 take $m>\abs{\alpha}/\grado$, and  define
		 \[
		 P_{m } f\coloneqq \sum_{h=0}^{2m-1} \frac{2^h}{h!} \Lc^h \ee^{-\Lc}f_0+\frac{1}{(2m-1)!} \sum_{j=1}^\infty \eps^{-j m} \int_{2 \eps^j}^{2 \eps^{j-1}} t^{2 m }\Lc^m \ee^{-(t-\eps^j)\Lc} f_{\eps^j} \frac{\dd t}{t}
		 \]
		 for every $f=(f_t)\in \Sc'(G)^{[0,+\infty)}$ for which the sum is defined, then $T_m(\Sc'(G))\subseteq \mathrm{dom} P_{m }$, $P_{m }T_{m } f=f$ for every $f\in \Sc'(G)$, and $\Pc_{m }\coloneqq T_{m } P_{m }$ is a projector of $\mathrm{dom}(P_{m })$ onto $\Fc_{m }$ which induces a continuous linear projector of $L^\infty(\beta)\oplus(\theta_\alpha \Cc^q(\mi,1/\grado))$ onto $[L^\infty(\beta)\oplus(\theta_\alpha \Cc^q(\mi,1/\grado))]\cap \Fc_{m }$;
		 
		 \item[\textnormal{(4)}] if we take $t_0\Meg 0$, $h\in\N$, and a minimal system $(X_j)_{j\in J}$ in $\gf$, with $d_j\coloneqq \deg(X_j)$ for every $j\in J$ and $t_0>0$ if $\alpha \meg 0$, then $f\in F^{\infty,q}_{\alpha+h\dd}(\beta)$ if and only if $\ee^{-t_0 \Lc} f\in L^\infty(\beta)$ and $X_j^{h \dd/d_j} f \in F^{\infty,q}_\alpha(\beta)$ for every $j\in J$;
		 
		 \item[\textnormal{(5)}] there is $\omega_0\in \R$ such that, for every $\alpha'\in \R$ and for every $\omega \Meg \omega_0$, the $(1+\abs{\gamma})^{-1/2} \ee^{-\pi\abs{\gamma}/2} \Lc_\omega^{(\alpha-\alpha')/\grado+i\gamma}$, $\gamma\in \R$, induce equicontinuous canonical isomorphisms of $F^{\infty,q}_\alpha(\beta)$ and $\mathring F^{\infty,q}_\alpha(\beta)$ onto $F^{\infty,q}_{\alpha'}(\beta)$ and $\mathring F^{\infty,q}_{\alpha'}(\beta)$, respectively.
	\end{enumerate}
\end{teo}

Notice that the proof of (3) fails when $\mi$ is not (essentially) a `dyadic' discrete measure.

\begin{proof}
	(1), (2), and (4) may be proved as the corresponding~\cite[Proposition 7.6, Lemma 7.5, and Theorem 7.8]{BCP}  with the previous substitutions. Notice, though, that in order to consider the case $\alpha>0$ and $t_0=0$ in (4), one has to use (3) in Proposition~\ref{prop:2b}. The reader may verify that the proof of this latter result does not rely on (4).
	
	(3) One may proceed as in the proof of the corresponding~\cite[Lemma 7.7]{BCP}  (with $\kappa=2$), with a modification (and a remark). The   modification  arises when trying to estimate $\norm{\Pc_{m,2}f}_{L^\infty(\beta)}$, where $\Pc_{m,2}=\Pc_m -\sum_{h=0}^{2m-1} \frac{2^h}{h!} \Lc^h \ee^{-2\Lc}$. In fact, in this case one would only get an estimate of the form (for suitable $b,C_1>0$)
	\[
	\begin{split}
		 \norm{\Pc_{m,2}f}_{L^\infty(\beta)}\meg C_1  \norm*{\norm{ s^{-\alpha/\grado}  ( T_{b,3} f_{s})(x)}_{L^q_s(\mi)}  }_{L^\infty_x(\beta)}  ,
	\end{split}
	\]
	which does not lead directly to the desired continuity. In order to conclude, one has to   observe that, by Lemma~\ref{lem:32}, there is a constant $C_2>0$ such that
	\[
	 (T_{b,3} f_{s})(x)\meg C_2 (T_{b/2,3}f_s)(y)
	\]
	for every $y\in B(x,1)$, so that there is a constant $C_3>0$ such that
	\[
	\begin{split}
	\norm*{\norm{ s^{-\alpha/\grado}   (T_{b,3} f_{s})(x)}_{L^q_s(\mi)}  }_{L^\infty_x(\beta)}&\meg C_3\norm*{\norm{ s^{-\alpha/\grado} \chi_{B(x,1)}(y)  (T_{b/2,3} f_{s})(y)}_{L^q_{(s,y)}(\mi\otimes\beta_L)}  }_{L^\infty_x(\beta)}\\
		&\meg C_3 \norm{s^{-\alpha/\grado} (T_{b/2,3} f_s)(y)}_{\Cc^q_{(s,y)}(\mi,1/\grado)},
	\end{split}
	\]
	from which one may conclude by means of Lemma~\ref{lem:2bis}. The remark (which explains the choice of $\mi$ discrete) occurs while estimating $\norm{[\Pc_{m,2 f}]_s(y)}_{\Cc^q_{(s,y)}(\mi,1/\grado)}$. In fact, in this case one would  get an estimate of the form (for suitable $b',C_4>0$)
	\[
	\begin{split}
		&\norm{ [\Pc_{m,2 }f]_s(y)}_{\Cc^q_{(s,y)}(\mi,1/\grado)}\meg C_4 \norm*{  s^{m-\alpha/\grado}  \sum_{ j=1 }^\infty  \frac{\eps^{m j}}{(\eps^j+s)^{2m}} T_{b',2 \eps^{j-1}} f_{\eps^j}  }_{\Cc^q_{(s,y)}(\mi,1/\grado)}.
	\end{split}
	\]
	In order to apply Lemma~\ref{lem:25b} we exploit the precise choice of $\nu$ to set $\kappa<\eps^{-1}$, so that the condition $f(t,\,\cdot\,)\meg c T_{b,ct}[g(s,\,\cdot\,)]$ for $(\mi\otimes \mi)$-almost every $(t,s)$ with $t/\kappa\meg s \meg t$ ($j\in\N$) reduces to the condition $f(\eps^j,\,\cdot\,)\meg c T_{b,ct}[g(\eps^j,\,\cdot\,)]$ for every $j\in\N$, which is readily met with $f(\eps^j,x)=g(\eps^j,x)=f_{\eps^j}(x)$ for every $j\in\N$ and $x\in G$, as well as suitable choices of $b,c>0$.  
	
	(5) This may be proved arguing as in the proof of~\cite[Theorem 7.9]{BCP}, except for the fact that we cannot apply Lemma~\ref{lem:25bisb} instead of~\cite[Lemma 6.6]{BCP} to 
	\[
	\norm*{\int_0^{+\infty} t^{m-\alpha'/\grado} s^{(\alpha'-\alpha)/\grado} \abs{(\ee^{-(t+s)\Lc}\Lc^m f)(x)}\,\frac{\dd s}{s}}_{\Cc^q_{(t,x)}(\mi_1,1/\grado)},
	\]
	but only to (say) 
	\[
	\norm*{\int_0^{1} t^{m-\alpha'/\grado} s^{(\alpha'-\alpha)/\grado} \abs{(\ee^{-(t+s)\Lc}\Lc^m f)(x)}\,\frac{\dd s}{s}}_{\Cc^q_{(t,x)}(\mi_1,1/\grado)}.
	\]
	In order to estimate the remaining term, it suffices to estimate
	\[
	\norm*{\int_1^{+\infty} t^{m-\alpha'/\grado} s^{(\alpha'-\alpha)/\grado} \abs{(\ee^{-(t+s)\Lc}\Lc^m f)(x)}\,\frac{\dd s}{s}}_{L^{q,\infty}_{t,x}(\mi_1,\beta)}.
	\]
	This may be done applying~\cite[Lemma 6.6]{BCP} to reduce to estimating
	\[
	\norm*{ \chi_{[1,\infty)}(t) t^{m-\alpha/\grado}   \abs{(\ee^{-t\Lc}\Lc^m f)(x)} }_{L^{q,\infty}_{t,x}(\mi,\beta)},
	\]
	were $\mi$ denotes a Haar measure on $(0,+\infty)$. This latter term is then controlled by
	\[
	\norm*{ \chi_{[1,\infty)}(t) t^{m-\alpha/\grado} \ee^{\omega_1(t-1/2)} (T^{\omega_1}_{b,*} \ee^{-(1/2)\Lc}f)(x)  }_{L^{q,\infty}_{t,x}(\mi,\beta)}
	\]
	for some $\omega_1<0$ and some $b>0$, with the notation of the proof of~\cite[Theorem 7.9]{BCP}. The assertion then follows from the fact that (we may assume that) $T^{\omega_1}_{b,*}$ is bounded on $L^\infty(\beta)$ (cf.~\cite[Lemma 5.10]{BCP} and its proof).   
\end{proof}

We now pass to the analogues of~\cite[Propositions 8.1 and 8.3]{BCP}. 

\begin{prop}\label{prop:1b}
	Take $q\in[1,\infty]$ and $\alpha\in \R$. Then,  $B^{\infty,q}_\alpha(\beta) \subseteq F^{\infty,q}_\alpha(\beta) \subseteq B^{\infty,\infty}_\alpha(\beta)$ continuously.
\end{prop}

\begin{proof}
	We may assume that $\beta=\beta_L$.
	Take $m\in\N$ with $m>\alpha/\grado$,   and observe that  there is a constant $C_1>0$ such that
	\[
	\begin{split}
		&t^{-Q_*/(q\grado )}\norm{ s^{-\alpha/\grado} \chi_{(0,t]\times B(x,t^{1/\grado})}(s,y) (W^{(m)}_{s} f)(y)  }_{L^q_{(s,y)}(\mi_1\otimes \beta)}\\
			&\qquad\meg t^{-Q_*/(q\grado )}\norm*{ s^{-\alpha/\grado} \norm{\chi_{B(x,t^{1/\grado}) }   W^{(m)}_{s} f }_{L^q(\beta)}  }_{L^q_{s}(\mi_1)}\\
			&\qquad\meg C_1 \norm*{ s^{-\alpha/\grado} \norm{ W^{(m)}_{s} f }_{L^\infty(\beta)}  }_{L^q_{s}(\mi_1)}
	\end{split}
	\]
	for every $f\in \Sc'(G)$, and for every $(t,x)\in (0,1] \times G$. The continuity of the inclusion  $B^{\infty,q}_\alpha(\beta) \subseteq F^{\infty,q}_\alpha(\beta)$ follows. 
	The proof of the reverse inclusion is essentially contained in the proof of Lemma~\ref{lem:27bis}. 
\end{proof}

\begin{prop}\label{prop:2b}
	Take $p,q_1,q_2\in[1,\infty]$ and $\alpha_1,\alpha_2\in \R$. Then, the following hold:
	\begin{enumerate}
		\item[\textnormal{(1)}] if either $\alpha_2<\alpha_1$ or $\alpha_2=\alpha_1$ and $q_1\meg q_2$, then
		\[
		F^{\infty,q_1}_{\alpha_1}(\beta)\subseteq F^{\infty,q_2}_{\alpha_2}(\beta)
		\]
		continuously;
		
		\item[\textnormal{(2)}] if   $p\in (1,\infty)$ and $\alpha_1-\frac{Q_*}{p}\Meg \alpha_2$, then
		\[
		\Delta_L^{1/p}F^{p,q_1}_{\alpha_1}(\beta)\subseteq F^{\infty,q_2}_{\alpha_2}(\beta)
		\]
		continuously;
		
		\item[\textnormal{(3)}] if  $\alpha_1>0$, then 
		\[
		F^{\infty,q_1}_{\alpha_1}(\beta) \subseteq L^\infty(\beta)
		\]
		continuously.
	\end{enumerate}
\end{prop}

\begin{proof}
	(1) The assertion follows from Proposition~\ref{prop:1b} and~\cite[Proposition 8.3]{BCP} when $\alpha_2<\alpha_1$. In fact, $F^{\infty,q_1}_{\alpha_1}(\beta) \subseteq B^{\infty,\infty}_{\alpha_1}(\beta)\subseteq B^{\infty, q_2}_{\alpha_2}(\beta)\subseteq F^{\infty,q_2}_{\alpha_2}(\beta)$ continuously. Then, we may assume that $\alpha_1=\alpha_2$ and that $q_1<q_2$. Observe first that the assertion follows from Proposition~\ref{prop:1b} when $q_2=\infty$ (since $F^{\infty,\infty}_{\alpha_1}(\beta)=B^{\infty,\infty}_{\alpha_1}(\beta)$). This means that there is a constant $C>0$ such that, for $m=[(\alpha_1/\grado)_+]+1$,
	\[
	\norm{t^{-\alpha_1/\grado} (W^{(m)}_t f)(x)}_{L^\infty_{(t,x)}(\mi_1\otimes\beta)}\meg C \norm{\ee^{-\Lc}f}_{L^\infty(\beta)}+C \norm{t^{-\alpha_1/\grado} (W^{(m)}_t f)(x)}_{\Cc^{q_1}_{(t,x)}(\mi_1,1/\grado)}.
	\]
	Consequently,
	\[
	\begin{split}
	&t^{-Q_*/\grado}\norm{s^{-\alpha_1/\grado} \chi_{(0,t]\times B(x,t^{1/\grado})}(W^{(m)}_s f)(y)}_{L^{q_2}_{(s,y)}(\mi_1\otimes \beta)}\\
	&\qquad\meg C^{1-q_1/q_2} ( \norm{\ee^{-\Lc}f}_{L^\infty(\beta)}+ \norm{t^{-\alpha_1/\grado} (W^{(m)}_t f)(x)}_{\Cc^{q_1}_{(t,x)}(\mi_1,1/\grado)}),
	\end{split}
	\]
	for every $x\in G$ and for every $t\in (0,1]$,
	whence the conclusion. 
	
	(2) Arguing as in the proof of (2$'$) of~\cite[Proposition 8.3]{BCP}, one may show that, taking $m\in\N$ with $m>\alpha_1/\grado$ and $f\in F^{p,\infty}_{\alpha_1}(\beta_L)$, there is a constant $C_1>0$ such that
	\[
	\sum_{j\Meg N} 2^{j\alpha_2/\grado}\abs{  W^{(m)}_{2^{-j}}f  }\meg C_1 2^{-NQ_*/(p\grado)} \sup_{j\in \N} 2^{j\alpha_1/\grado}\abs{  W^{(m)}_{2^{-j}}f  },
	\]
	for every $N\in\N$,
	so that there is a constant $C_2>0$ such that
	\[
	\begin{split}
	&2^{NQ_*/\grado} \int_{B(x,2^{-N/\grado})} \sum_{j\Meg N} 2^{j\alpha_2/\grado}\abs{  (W^{(m)}_{2^{-j}}f )(y) }\,\dd \beta(y)\\
		&\qquad\meg  C_1 2^{-NQ_*/(p\grado  )}  2^{N Q_*/\grado} \int_{B(x,2^{-N/\grado})} \sup_{j\in \N} 2^{j\alpha_1/\grado}\abs{ ( W^{(m)}_{2^{-j}}f)(y)  }\,\dd \beta(y)\\
		&\qquad\meg C_2  \norm*{\sup_{j\in \N} 2^{j\alpha_1/\grado}\abs{ ( W^{(m)}_{2^{-j}}f)(y)  }  }_{L^p_y(\beta_L)}.
	\end{split}
	\]
	The continuity of the inclusion $F^{p,\infty}_{\alpha_1}(\beta_L)\subseteq F^{\infty,1}_{\alpha_2}(\beta)$ then follows from the arbitrariness of $x\in G$ and $N\in\N$. The general case then follows, thanks to (1) and Proposition~\ref{prop:5}.
	
	(3) This follows from (1) and~\cite[Proposition 8.3]{BCP}, using the equality $F^{\infty,\infty}_\alpha(\beta)=B^{\infty,\infty}_\alpha(\beta)$.
\end{proof}

We observe explicitly that studying the interpolation spaces of the `full' scale of the Triebel--Lizorkin spaces $F^{p,q}_\alpha(\beta)$, $p\in (1,\infty]$, $q\in[1,\infty]$, $\alpha\in\R$, seems to pose substantial difficulties in this generality. Apart from the analogue of  (5) of~\cite[Theorem 10.1]{BCP}, which may be proved with similar arguments (replacing~\cite[Proposition 8.3]{BCP} with Proposition~\ref{prop:2b}), the other results are substantially harder to deal with. For example, one may try to prove an equality of the form: $(\mathring F^{p_0,q_0}_{\alpha_0}(\beta), \mathring F^{\infty,q_1}_{\alpha_1}(\beta))_{[\theta]}=\mathring F^{p_\theta,q_\theta}_{\alpha_\theta}(\beta)$ for $p_0<\infty$ and suitably defined $p_\theta,q_\theta$, and $\alpha_\theta$. On the one hand, we may prove the inclusion $\supseteq$ directly, using the standard interpolation procedure for mixed norm spaces, and using (2) and (3) of Theorem~\ref{teo:15} (and the corresponding~\cite[Lemmas 7.5 and 7.7]{BCP}). On the other hand, for the inclusion $\subseteq$, we   incur in the non-trivial problem of `interpolating' between a `full' mixed norm $L^{q_0,p_0}$ and a `localized' norm $\Cc^{q_1}$. The techniques developed in~\cite{FrazierJawerth} cannot be applied here since we lack a reasonable `discretization' (which would allow us to study Calder\'on products instead of complex interpolation) and the `un-discretized' results are too weak to help us in this situation.  
 
We then pass to the algebra properties and pointwise multipliers. The following result in an analogue of~\cite[Theorem 3.1]{Calzi2}.

\begin{teo}\label{teo:6b}
	Take $\alpha>  0$, $p,p_1,p_2,p_3,p_4\in (1,\infty]$, and $q\in [1,\infty]$ such that $\frac{1}{p_1}+\frac{1}{p_2}=\frac{1}{p_3}+\frac{1}{p_4}=\frac{1}{p}$. 
	Then, there is a constant $C>0$ such that, for every $f\in F^{p_1,q}_\alpha(\beta)\cap L^{p_3}(\beta)$ and for every $g\in F^{p_4,q}_\alpha(\beta)\cap L^{p_2}(\beta)$,
		\begin{equation*}
			\norm{f g}_{F^{p,q}_\alpha(\beta)}\meg C \norm{f}_{F^{p_1,q}_\alpha(\beta)}\norm{g}_{L^{p_2}(\beta)}+ C \norm{f}_{L^{p_3}(\beta)} \norm{g}_{F^{p_4,q}_\alpha(\beta)}.
		\end{equation*}
\end{teo}
  
\begin{proof}
	Observe first that, by~\cite[Theorem 3.1]{Calzi2}, we may reduce to the case in which $p_1=\infty$ or $p_4=\infty$. By symmetry, we may in fact assume that $p_4=\infty$, in which case   $p_3=p$. 	
	If $p=\infty$, then $p_1=p_2=p_3=p_4=\infty$ and the proof proceeds following the same arguments of the proof of~\cite[Theorem 3.1]{Calzi2}, with the usual replacements. More precisely, one may apply Lemma~\ref{lem:25b}, with any $\kappa\in (1/2,1)$, to $f(t,x)=T_{b,t}[(W^{(k)}_t f)(W^{(h)}_t g)](x)$ and $g(t,x)=(T_{b',t}W^{(k)}_{t/2} f)(x)(T_{b',t}W^{(h)}_{t/2} g)(x)$ for  a suitable $b'>0$ (for every $h,k=0,\dots, m$), and then proceed with similar estimates (cf.~also the modifications described below).
	Then, assume that $p<\infty$. If $p_1<\infty$, then the estimation of $\Pi_f^{(t')} g$ proceeds unchanged. For what concerns the estimation of $\Pi_g^{(t')} f$, one may first observe that, by Theorem~\ref{teo:7}, there are $b,C_1>0$ such that $\abs{  \Lc^k h_t}\meg C_1 p_{b,t}$ for every $k=0,\dots, m$  and for every $t\in (0,1]$, and then observe that, by Lemmas~\ref{lem:59} and~\ref{lem:60}, there are a $(\mi_1\otimes \beta)$-measurable subset $E$ of $(0,+\infty)\times G$ and  a constant $C_2>0$ such that
	\[
	\begin{split}
		&\sum_{k=0}^{m-1}\norm*{ \norm*{ t^{-\alpha/\grado}  (W^{(k)}_t f)(x) (W^{(m)}_t g)(x)   }_{L^q_t(\mi_1)}}_{L^p_x(\beta)}\\
		&\qquad \meg C_1^2 \sum_{k=0}^{m-1}\norm*{  \norm*{ t^{-\alpha/\grado}  (T_{b,t/2} W^{(k)}_{t/2} f)(x) (T_{b,t/2}W^{(m)}_{t/2} g)(x)   }_{L^q_t(\mi_1)}}_{L^p_{x}( \beta)}\\
		&\qquad  \meg  C_1^2  C_2 \sum_{k=0}^{m-1}\norm*{  \norm*{ t^{-\alpha/\grado} \chi_E(t,x) (T_{b/2,t/2} W^{(k)}_{t/2 } f)(x) ( T_{b/2,t/2}W^{(m)}_{t/2} g)(x)   }_{L^q_t(\mi_1)}}_{L^p_{x}( \beta)}
	\end{split}
	\]
	and
	\[
	\norm{\chi_E(t,x) t^{-\alpha/\grado} (T_{b/2,t/2}W^{(m)}_{t/2} g)(x)  }_{L^{q,\infty}_{t,x}(\mi_1,\beta)}\meg C_2\norm{t^{-\alpha/\grado}( T_{b/2,t/2}W^{(m)}_{t/2} g)(x)}_{\Cc^q_{(t,x)}(\mi_1,1/\grado)}.
	\]
	Now, by Lemma~\ref{lem:4} there is a constant $C_3>0$ such that
	\[
	T_{b/2,t/2}W^{(k)}_{t/2} f \meg C_1 T_{b/2,t/2}T_{b/2,t/2} f\meg C_3 T^*_{b/4,1} f
	\]
	for every $t\in (0,1]$, while 
	by Lemmas~\ref{lem:2bis},~\ref{lem:61}, and~\ref{lem:8b}  there is a constant $C_4>0$ such that
	\[
	\norm{(T_{b/2,t/2}W^{(m)}_{t/2} g)(x)}_{\Cc^q_{(t,x)}(\mi_1,1/\grado)}\meg  C_4\norm{g}_{F^{\infty,q}_\alpha(\beta)}.
	\]
	It then follows that
	\[
	\sum_{k=0}^{m-1}\norm*{  \norm*{ t^{-\alpha/\grado} \chi_E(t,x) ( T_{b/2,t/2}W^{(k)}_{t/2 } f)(x) (T_{b/2,t/2} W^{(m)}_{t/2 } g)(x)   }_{L^q_t(\mi_1)}}_{L^p_{x}( \beta)}\meg  m C_3 C_4\norm{T^*_{b/4 ,1} f}_{L^p(\beta)} \norm{g}_{F^{\infty,q}_\alpha(\beta)},
	\]
	whence the desired estimation by~\cite[Lemma 5.10]{BCP}. If, otherwise, $p_1=\infty$, then also the estimation of $\Pi_f^{(t')} g$ should be modified as above.
	The estimation of $\Pi(f,g)$ proceeds with similar modifications. 
\end{proof}

\begin{cor}
	Take $q\in [1,\infty]$ and $\alpha>0$. Then, the space of pointwise multipliers of $F^{\infty,q}_\alpha(\beta)$ (that is, the space of  $f\in \Sc'(G)$ such that the mapping $\Sc(G)\ni \phi\mapsto f \phi\in \Sc'(G)$ induces a continuous linear mapping $\mathring F^{\infty,q}_\alpha(\beta)\to F^{\infty,q}_\alpha(\beta)$) is the space $F^{\infty,q}_\alpha(\beta)$ itself.
\end{cor}

\begin{proof}
	The assertion follows from Theorem~\ref{teo:6b} and the fact that $F^{\infty,q}_\alpha(\beta)$ contains the function identically equal to $1$, arguing as in the proof of~\cite[Proposition 6.2]{Calzi2}.
\end{proof}

We then have the analogues of the results concerning localization (\cite[Propositions 4.1 and 4.4]{Calzi2}), which may be proved with similar techniques.

\begin{prop}\label{prop:34b}
	Let $\phi$ be a function of class $C^\infty$ on $G$, and assume that there are $c,\eta>0$  such that  $\abs{\phi(x)-1}\meg c\abs{x}_*^\eta \ee^{c\abs{x}_*}$ for every $x\in G$.
	Take $q\in [1,\infty]$ and $\alpha\in (0,\eta)$. In addition, take $m\in \N$ such that $m>\alpha/\grado$, and take $\mi\in \Mc_\Samp$. Then, there is a constant $C>0$ such that, for every $f\in \Sc'(G)$,
		\[
		\frac{1}{C}\norm{f}_{F^{\infty,q}_\alpha(\beta)}\meg \norm{f}_{L^\infty(\beta)}+ \norm{ t^{m-\alpha/\grado} [f*(\phi \Lc^m h_t)](x) }_{\Cc^{q}_{(t,x)}(\mi,1/\grado)}\meg C \norm{f}_{F^{\infty,q}_\alpha(\beta)}.
		\]
\end{prop}

\begin{prop}\label{prop:30b}
	Take   $q\in [1,\infty]$   and $\alpha>0$. Take $\delta>0$, $R\Meg 2$, and a $(\delta,R)$-lattice $(x_j)_{j\in J}$ on $G$.  
	Take two bounded families $(\phi_j)_{j\in J}$ and $(\psi_j)_{j\in J}$ of elements of $C^\infty_c(G)$. 
	Then, the continuous linear mappings
	\[
	\Ic_{(\phi_j)}\colon \Dc'(G)\ni u\mapsto (\phi_j(x_j^{-1}\,\cdot\,) u)\in \Dc'(G)^J
	\]
	and
	\[
	\Rc_{(\psi_j)}\colon \Dc'(G)^J \ni  (u_j) \mapsto \sum_{j\in J} \psi_j(x_j^{-1}\,\cdot\,) u_j \in \Dc'(G)
	\]
	induce continuous linear mappings 
	\[
	\begin{aligned} 
		F^{\infty,q}_\alpha(\beta)&\to \ell^\infty(J;F^{\infty,q}_\alpha(\beta)), & \mathring F^{\infty,q}_\alpha(\beta)&\to \ell^\infty(J;\mathring F^{\infty,q}_\alpha(\beta)), 
	\end{aligned}
	\]
	and
	\[
	\begin{aligned}
		\ell^\infty(J;F^{\infty,q}_\alpha(\beta))&\to F^{\infty,q}_\alpha(\beta),& \ell^\infty(J;\mathring F^{\infty,q}_\alpha(\beta))&\to \mathring F^{\infty,q}_\alpha(\beta), 
	\end{aligned}
	\] 
	respectively.
	If, in addition, $\sum_j (\phi_j\psi_j)(x_j^{-1}\,\cdot\,)=1$, then $\Rc_{(\psi_j)} \Ic_{(\phi_j)}=I$.
\end{prop}

Notice that, since we cannot argue by duality, we limit ourselves to $\alpha>0$.

\begin{proof}
	The continuity of $\Ic_{(\phi_j)}$ follows easily from Theorem~\ref{teo:6b}. The continuity of $\Rc_{(\psi_j)}$ may be proved as in the proof of~\cite[Proposition 4.4]{Calzi2}. 
\end{proof}

  Concerning the characterization with differences, we have not been able to extend the proof of\cite[Proposition 7.8]{Calzi2} to this context.

\end{document}